\documentclass[11pt]{amsart}

\usepackage{graphicx}
\usepackage{color}
\usepackage{pst-grad} 
\usepackage{pst-plot} 
\usepackage{pstricks-add}
\usepackage{algorithm}
\usepackage{subfigure}
\usepackage{amsmath,amsthm}
\usepackage{amssymb,enumerate}
\usepackage[T1]{fontenc}

\newcommand{\rr}{\mathbb{R}}

\newcommand{\dd}{\mathrm{d}}

\newcommand{\R}{\mathbb{R}}

\newcommand{\Div}{\mathrm{div}}
\newcommand{\argmin}{\operatorname{argmin}}

\newcommand{\norm}[1]{\Vert #1 \Vert}

\newcommand{\exit}{\partial\Omega_{\text{exit}}}
\newcommand{\wall}{\partial\Omega_{\text{wall}}}

\newcommand{\psxinput}[2]{\psscalebox{#2 #2}{#1}}
\theoremstyle{remark}
\newtheorem*{remark}{Remark}
\newtheorem{lemma}{Lemma}
\newtheorem{definition}{Definition}

\begin{document}

\markboth{Carrillo, Martin, Wolfram }{A local version of the Hughes model for pedestrian flow}

%
%

\title{A local version of the Hughes model for pedestrian flow }

\author{Jose A. Carrillo$^*$}


\thanks{$^*$
Department of Mathematics, Imperial College London, London SW7 2AZ, UK\\
carrillo@imperial.ac.uk}

\author{Stephan Martin$^\text{\textdagger}$}
\thanks{$^\text{\textdagger}$
Department of Mathematics RWTH Aachen University, 52074 Aachen, Germany \& \\
Department of Mathematics, Imperial College London, London SW7 2AZ, UK\\
martin@mathcces.rwth-aachen.de, stephan.martin@imperial.ac.uk, }

\author{Marie-Therese Wolfram$^\ddagger$}
\thanks{$^\ddagger$
Radon Institute for Computational and Applied Mathematics, \\
Austrian Academy of Sciences, Altenberger Strasse 69, 4040 Linz, Austria\\
mt.wolfram@ricam.oeaw.ac.at
}



\maketitle

\begin{abstract}
Roger Hughes proposed a macroscopic model for pedestrian dynamics, in which individuals seek to minimize their travel time but
try to avoid regions of high density. One of the basic assumptions is that the overall density of the crowd is known to every agent. In this paper we present a modification of the Hughes model to include local effects, namely limited vision, and a conviction towards decision making. The modified velocity field enables smooth turning and temporary waiting behavior. We discuss the modeling in the micro- and macroscopic setting as well as the efficient numerical simulation of either description. Finally we illustrate the model with various numerical experiments
and evaluate the behavior with respect to the evacuation time and the overall performance.
\end{abstract}



\section{Introduction}
The mathematical modeling and simulation of pedestrian dynamics,
such as large human crowds in public space or buildings, has
become a topic of high practical relevance. The complex behavior
of these large crowds poses significant challenges on the
modeling, analytic and simulation level. These aspects initiated a
lot of research in the mathematical community within the last
years, which we briefly outline below. Mathematical modeling
approaches for pedestrian dynamics can be roughly grouped into the
following categories:
\begin{enumerate}
\item \textit{Microscopic models} such as the social force model\cite{helbing1995,helbing2000,treuille2006} or
cellular automata approaches\cite{burstedde2001}.
\item \textit{Fluid dynamic approaches}\cite{colombo2012,moussaid2012,appert2011} and related \textit{macroscopic models}, see for example the popular Hughes model\cite{hughes2002,difrancesco2011,goatin2013,burger2014,VenutiBruno}.
\item \textit{Kinetic models}  \cite{bellomo2013,DARPT2013} which uses ideas from gas kinetics to models interactions between individuals via so-called collisions.
\item In \textit{optimal control}\cite{hoogendoorn2004} and \textit{mean field game approaches}\cite{lachapelle2011,dogbe2010} pedestrians act as rational individuals, which adjust
their velocity optimal to a specific cost.
\item \textit{Multiscale models} coupling between different scales to describe for example crowd leader dynamics \cite{CPT2011,BPT2012}.
\end{enumerate}
A detailed survey on crowd modelling can e.g. be found in Bellomo
and Dobge\cite{BellomoDogbe2011}. Several aspects are considered
to be important in the mathematical modeling to capture the
complex behavior in a correct way. For example repulsive forces
when getting too close to other individuals or obstacles play an
important role in the dynamics. Another popular assumption is the
fact that individuals act rationally and try to make the optimal
decision based on their actual knowledge level. Partial knowledge
of the overall pedestrian density or the domain is another
important factor which should be taken into account in the
modelling. While these nonlocal effects can be implemented quite
intuitively on the microscopic level, their translation for
macroscopic models is not straightforward. Most macroscopic
nonlocal models are based on the continuity equation for the
pedestrian density, where the nonlocal effects correspond to the
deviation of the crowd from its preferred
direction\cite{colombo2012,colombo2012-2,crippa2013}. This
deviation is determined by the average density felt by the
pedestrians and modelled via a convolution operator acting on the
velocity. The development of numerical schemes for conservation
laws with nonlocal effects gained substantial interest in the last
years. This was, among other factors, also initiated by the
development of nonlocal models in traffic
flow \cite{ACT2014,BG2014}.   \\
\noindent The original model of Hughes\cite{hughes2002} describes
fast exit and evacuation scenarios, where a group of people wants
to leave a domain $\Omega \subset \mathbb{R}^2$ with one or
several exits/doors and/or obstacles as fast as possible. The
driving force towards the exit is the gradient of a potential
$\phi = \phi(x,t)$, $x \in \Omega,~t > 0$. This potential
corresponds to the expected travel time to manoeuvrer through the
present pedestrian density towards an exit. Hughes assumed that
the global distribution of pedestrians is known to every
individual, an assumption not generally satisfied in real world
applications.
\\
In this paper we present a generalization of the classical Hughes
model, which includes local vision via partial knowledge of the
pedestrian density. We discuss the proper modeling setup, the
implementation of suitable numerical schemes as well as their
computational complexity. Furthermore we compare how the reduced
perception of each pedestrian effects the overall "performance" of
the crowd in evacuation scenarios. Inevitably, one expects the
crowd to behave less efficient as less information is available.
Quantifying how localised vision influences performance and
decision making is a very interesting question in terms of
collective behaviour. Surprisingly, it will turn out that
evacuation times can even improve. The question we investigate is
therefore complementary to mean-field game approaches, where
pedestrians anticipate future crowds states and hence are
\emph{more} capable than in the original Hughes' model
\cite{lachapelle2011,dogbe2010,burger2014}.
\\

\noindent This paper is structured as follows: We start with a
review on the modeling and analytic results of the classical
Hughes model for pedestrian flow and its microscopic
interpretation in section \ref{s:hughes}. In section
\ref{s:mathmod} we present the local version of the Hughes model
on the micro- and macroscopic level.  Section \ref{s:computmeth}
presents the numerical strategies for the microscopic and
macroscopic model. We compare the behavior and performance of the
models in Section \ref{s:numexp} and conclude with a discussion of
the proposed model in Section \ref{s:conclusions}.

\section{Hughes' model for pedestrian flow}\label{s:hughes}
\subsection{Original formulation and analytic results}

Let us start by presenting the original modeling assumptions and
the corresponding partial differential equation system of the
Hughes model for pedestrian flow. Hughes considered an exit
scenario, in which a crowd modelled by a macroscopic density $\rho
= \rho(x,t)$ wants to leave a domain as fast as possible. The
nonlinear PDE system for $\rho$ and the potential $\phi =
\phi(x,t)$ on the domain $\Omega \subset \R^2$ reads as:
\begin{subequations}\label{e:hughes}
\begin{align}
\frac{\partial \rho}{\partial t}-\Div(\rho f(\rho)^2 \nabla \phi) &= 0 \label{e:hughescont}\\
\norm{\nabla \phi} &= \frac{1}{f(\rho)}. \label{e:hugheseikonal}
\end{align}
\end{subequations}
The first equation describes the evolution of $\rho$ in time,
driven by the gradient of $\phi$ and weighted by a nonlinear
mobility $f = f(\rho)$. This mobility includes saturation effects,
i.e. degenerate behaviour when approaching a given maximum density
$\rho_{\max} \in \R^+$. Possible choices are $f(\rho) =
\rho-\rho_{\max}$ or $f(\rho) = (\rho-\rho_{\max})^2$ amongst
others. The former is inherited from the
Lighthill-Whitham-Richards model for one-dimensional traffic flow
\cite{LW1955,R1956}.
\\
The potential $\phi$ corresponds to the weighted shortest distance
to an exit in the following sense: Solving the eikonal equation
\eqref{e:hugheseikonal} determines the optimal path $\nabla \phi$
minimising the expected travel time throughout the crowd towards
an exit. This cost is measured as the inverse of $f(\rho)$, hence
the cost of walking through dense regions is high. Equation
\eqref{e:hugheseikonal} is also a stationary
Hamilton-Jacobi-Bellman equation, and the optimal path property of
$\nabla\phi$ can be rigorously derived \cite{BC,Holm}. The fact
that the potential $\phi$ solely determines the direction of the
flow can be easily seen as $f^2(\rho)\nabla\phi =
f(\rho)\nabla\phi / \norm{\nabla\phi}$ using
\eqref{e:hugheseikonal}.
\\

\noindent Hughes model \eqref{e:hughes} is supplemented with
different boundary conditions for the walls and the exits. We
assume that the boundary is divided into two parts: either
impenetrable walls $\wall \subset \partial \Omega$ or exits/doors
$\exit \subset \partial \Omega$, with $\wall \cap \exit =
\emptyset$. Typical conditions for the density $\rho$ in
\eqref{e:hughes} are zero flux boundary conditions at $\wall$,
which are either automatically satisfied as $\nabla\phi \cdot \vec
n =0$ or artificially enforced. The flux at $\exit$ has to be
defined according to the arriving density and our choices are
discussed in Section \ref{s:mathmod}. The boundary conditions of
\eqref{e:hugheseikonal} are set as $\phi(x,t) = 0$ for all $ x \in
\exit$.

\noindent There has been a lot of recent mathematical research on
the classical Hughes model
\cite{difrancesco2011,goatin2013,amadori2012,elkhatib2013}. Up to
the authors knowledge all analytic results are restricted to 1D
only, which is caused by the low regularity of the potential
$\phi$. This low regularity, i.e. $\phi \in C^{0,1}$, results in
the formation of shocks and rarefaction waves in the conservation
law. It is caused particularly by the existence of \emph{sonic
points}, which are hypersurfaces in space, where costs towards two
or more exists coincide, and therefore $\nabla\phi$ does not exist
and the orientational field is discontinuous. In spatial dimension
one the system can be reduced to the conservation law with a
discontinuous flux function. In this case it is possible to solve
the corresponding Riemann problem \cite{difrancesco2011}, which
also serves as a basis for different numerical schemes
\cite{goatin2013,BG2014}.

\subsection{Microscopic interpretation}

Hughes motivated system \eqref{e:hughes} on the macroscopic level
only. Recently Burger et al. \cite{burger2014} were able to give a
microscopic interpretation of \eqref{e:hughes}, which will serve
as a basis for our local particle model. Microscopic models based
on Hughes' modelling assumptions are also used in the field of
computer vision \cite{treuille2006}.

\noindent Let us consider $N$ particles with position $X^j =
X^j(t)$ and velocity $V^j = V^j(t)$, $j=1,\ldots N$. Then the
empirical density $\rho^N = \rho^N(t)$ is given by
\begin{align}\label{e:rhon}
\rho^N(t) =  \frac{1}{N} \sum_{j=1}^N \delta(x-X^j(t)).
\end{align}
Furthermore we introduce its smoothed approximation $\rho^N_g = \rho^N_g(t)$, given by
\begin{align*}
\rho^N_g(x,t) = (\rho^N * g)(x,t) = \frac{1}{N} \sum_{j=1}^N  g(x-X^j(t)),
\end{align*}
where the function $g = g(x)$ corresponds to a sufficiently smooth
positive kernel. The walking cost is given by the sum of a
weighted kinetic energy and the exit time, defined as $T_{exit} =
\sup\lbrace t>0 \mid x \in \Omega \rbrace$. Then the problem reads
as:
\begin{align}\label{e:minC}
\begin{split}
&C(X;\rho(t))  = \min_{(X,V)} \frac{1}2\int_t^{T+t} \frac{\norm{V(s)}^2}{f^2(\rho^N_g(\xi(s;t),t))} ds + \frac{1}2 T_{exit}(X,V),\\
&\text{ subject to } ~~\frac{d \xi}{ds}=V(s) \text{ and }
\xi(0)=X(t).
\end{split}
\end{align}
Hence the optimal trajectory is determined by 'freezing' the
empirical  density $\rho^N = \rho^N(t)$, in other words it
corresponds to extrapolating the empirical density $\rho^N$
in time when looking for the optimal trajectory.
\\
Burger et al. were able to show  that Hughes' model can be
formally derived from the optimality conditions of \eqref{e:minC}
and letting $T \rightarrow 0$ (corresponding to the long-time
behavior of the corresponding adjoint Hamilton-Jacobi equation).
\\
We will use this microscopic interpretation to propose a numerical
approximation by a particle method in Section 4 of Hughes type
models. In fact, \eqref{e:hughescont} is seen as a continuity
equation with velocity field $v(x,t)=-f(\rho)^2 \nabla \phi$
driven by \eqref{e:hugheseikonal}, and thus particles in
\eqref{e:rhon} are advected by the velocity field $v$, e.g.
$$
\frac{d X^j}{dt}=v(X^j(t),t)\,,\qquad j=1,\dots N\,.
$$


\section{A localized smooth Hughes-type model for pedestrian flow}\label{s:mathmod}

\noindent The Hughes model \eqref{e:hughes} assumes that at any
time $t>0$ the global distribution of all other individuals $\rho(x,t)$ is known to every
pedestrian. Therefore she chooses her optimal walking direction $\nabla
\phi$ in order to minimise its expected travel time/costs. Here,
all walking costs are based on the current density, which means
that pedestrians do not anticipate future dynamics of the crowd.
Instead they are capable to react to changes in the global density
ad-hoc as the path optimisation is repeated continuously in time.
In a mean-field game type model, the capabilities of pedestrians
would increase, as the planning decision of all agents can be
correctly predicted into the future. We follow an opposite
approach and reduce the capabilities of pedestrians, to obtain a
more realistic model.
\\
\noindent The assumption of continuous and complete perception of
global density information at current time is highly questionable
in practical situations. Limited vision cones and restricted
perception of global information comes through obstacles (walls,
buildings), physical distance, visual orientation or the inability
to see through a very dense crowd. Some effects of local vision
on the behaviour of crowds are obvious: in an evacuation scenario
with two exit corridors, which cannot be seen from each other,
pedestrians caught in a jam in front of one exit will not be able
to see whether the other exit is free or also jammed.

\noindent These considerations motivated a new version of
Hughes-type pedestrian dynamics based on localised perception of
information, which we introduce in this section. The decision of
each pedestrian is based on the perceived local density available
in a limited domain, which can be e.g. interpreted as a vision
cone. Furthermore
 a local interaction mechanism between individuals as well as a smoothening kernel
on the velocity field (to prevent unrealistic high frequency
oscillations in the direction of motion) are incorporated. We
begin with the detailed introduction of the macroscopic model and
discuss the microscopic analogue thereafter.
\subsection{Macroscopic equations}
The starting point of our model is the assumption that pedestrians
still perform the same path-optimisation selection as in the
Hughes' model, while the crowd state they act upon subjectively
depends on their position and the amount of information they are
able to perceive. Let $y\in\rr^2$ be an auxiliary variable and
$\phi(x,y):\rr^4\rightarrow\rr$ a parameterised potential, such
that
$
y\mapsto \phi(x_0,y)
$ 
denotes the cost potential calculated by pedestrians located at
$x_0\in\Omega$. To model space-dependent perception of
information, suppose that for every $x$ the domain $\Omega$
decomposes into a visible subdomain $V_x\ni x$ and a hidden or
invisible part  $H_x = \Omega \backslash V_x$. We propose the
following mechanism of restricted vision: If an area is visible
its density is known and priced accordingly in the path
optimisation. If however an area is not visible, its density is
thought to be a constant $\rho_H \in \R^+_0$, which we assume to
be uniform among all pedestrians. Exemplarily, setting $\rho_H=0$
implies that pedestrians assume that not visible areas to them are
empty, hence they will have a strong incentive to explore unseen
parts of the domain. On the contrary, pedestrians will avoid
invisible areas when $\rho_H\approx \rho_{\max}$, as they assume
high costs. An eikonal equation in $\Omega$ is hence solved in the
auxiliary variable $y$ for every point $x$ as
\begin{equation}
\norm{\nabla_y \phi(x,y )} = \displaystyle{\begin{cases}
 \frac{1}{f(\rho(y,t))} & y\in V_x \\[1mm]
 \frac{1}{f(\rho_H)} & y\in H_x
 \end{cases}}, 
 \label{model-local-eikonal}
\end{equation}
which gives the potential $\phi$ as function of two space
variables. Note, that this notion of local perception differs from
other recent work\cite{BG2014}, where a local average of the
density is used. Each pedestrian uses the cost potential at her
own position for the decision process. Computing $\nabla_y
\phi(x,x)$ hence would, after normalisation, give a new
orientational vector field to be used in the unchanged transport
equation. We however argue that it makes sense to include a notion
of conviction to the model, which has previously not been
considered. In order to do so, \eqref{model-local-eikonal} is
solved for every single exit. This results in the computation of
$M$ potentials $\phi_k=\phi_k(x,t)$, $k=1,\dots, M$, which allows
for cost comparison between exits, see Remark 3.2.

In regions of high density, decisions on the walking direction
towards any of $k=1,\dots,M$ exits $\partial\Omega_{E_k}$ cannot
arbitrarily deviate between neighbours. If a pedestrians prefers
to walk against the direction of a predominant local flow,
collision or friction losses in the movement will occur.
Especially on the macroscopic level, in which we take an aggregate
perspective, an incentive to change the flow cannot arise from one
point in space alone. As we do not model the granular level of
individual collision or friction, we propose the following
mechanism:
\begin{enumerate}
\item Each pedestrian carries an individual conviction strength $u(x)$
measuring its preference of its chosen exit over all others.
\item There exists a local consensus process within the crowd, which results
in the adjustment of the individual walking direction according to the
predominant direction around them.
\end{enumerate}
Hence, pedestrians adjust their own direction in order to prevail
the flow rather than obstructing it. This can be seen as either a
cognitive decision rule or a forced physical restriction. For a
compactly supported interaction kernel
$\mathcal{K}:\rr^2\rightarrow \rr$, we define the final walking
direction $\varphi(x)$ at any point $x\in\Omega$ as
\begin{equation}
\varphi = \frac{\rho u \star \mathcal{K}}{\rho\star\mathcal{K}},
\end{equation}
where the conviction $u(x)$ is given as
\begin{equation}
u =  \frac{\nabla_y \phi_{k^\text{opt}}}{\norm{\nabla_y
\phi_{k^\text{opt}}}}(\phi_{k^{\text{opt}+1}}-\phi_{k^\text{opt}}),
\end{equation}
obtained by comparing the cost potentials $\phi_k$, $k=1,\dots,
M$, associated to each of the exits:
\begin{gather}
 k^{\text{opt}}(x) = \argmin\limits_k \phi_k(x,x), \\
 k^{\text{opt+1}}(x) = \argmin\limits_{k\neq k^{\text{opt}}} \phi_k(x,x).
\end{gather}
Discontinuities in the velocity field due to the heterogeneity of
decision making amongst pedestrians are hence partially
compensated. To further smooth the model, we relax the strict
restriction of $\norm{\nabla\varphi}=1$ of Hughes' model and
replace the normalisation operator with a smooth approximation
$\mathcal{P}:\rr^2\rightarrow \rr^2$ defined as:

\begin{equation}
\mathcal{P}[x] := \begin{cases}
\frac{x}{\norm{x}} & \norm{x}>\ell, \\
\sin\left(\frac{\pi}{2 \arctan(k\ell)}\arctan(k\norm{x})\right)\frac{x}{\norm{x}} & 0<\norm{x}\leq\ell, \\
{0} & x=0 ,
\end{cases},
\label{e-relaxation}
\end{equation}
for some parameters $k,\ell>0.$ We stress that this is not a
technicality, as we here allow pedestrians to \emph{stop when
being undecided}. This is highly desirable from the modelling
point of view, though on the other hand the modulus of the flux
now is not a function of density alone, as one can see below.
\\
Next we discuss the boundary conditions for the eikonal equations.
Since we treat each exit separately, we set
$\left.\phi_k\right\vert_{\partial \Omega_{E_k}} = 0$ in the
computation of $\phi_k$. No boundary conditions are imposed on the
rest of the boundary  $\wall $.

\noindent Near-wall and near-obstacle effects have a strong
influence on the dynamics on constrained macroscopic evolutions.
We propose that pedestrians take into account walls and obstacles
in their computation of optimal paths. Hence it is natural to
include these effects as an additional fixed cost $W(x)$ on the
right-hand side of the eikonal equation
\eqref{model-local-eikonal}. We introduce a smooth layer profile
$\chi_w(x)\in[0,1]$, which identifies areas close to walls but
smoothly vanishes elsewhere and around exits to allow outflow. A
typical choice of $\chi_w$ is illustrated in Figure
\ref{fig-layerprofile}. For the sake of simplicity, we set
\begin{equation}
W(x) = \frac{\chi_w(x)}{f(\rho_{\max}-\epsilon)},
\label{e:wall}
\end{equation}
hence areas close to walls are penalised similar to high density areas.
\begin{figure}
\centering \psxinput{{
\begin{pspicture}(0,-1.7848612)(10.055555,1.7848612)
\psline[linecolor=black, linewidth=0.04](0.15555558,-1.0440277)(9.955556,-1.0440277)
\psline[linecolor=black, linewidth=0.04](0.15555558,-1.1440277)(9.955556,-1.1440277)
\psline[linecolor=black, linewidth=0.04](3.8555555,-1.0440277)(0.65555555,1.3559723)
\psline[linecolor=black, linewidth=0.04](5.8555555,-1.0440277)(9.055555,1.3559723)
\psline[linecolor=black, linewidth=0.04, linestyle=dashed, dash=0.17638889cm 0.10583334cm](3.8555555,-1.0440277)(3.8555555,-1.8440278)
\psline[linecolor=black, linewidth=0.04, linestyle=dashed, dash=0.17638889cm 0.10583334cm](5.8555555,-1.0440277)(5.8555555,-1.8440278)
\rput[bl](4.5555556,-1.4440278){exit}
\rput[bl](7.0555553,-1.4440278){wall}
\rput[bl](1.8555556,-1.4440278){wall}
\psline[linecolor=black, linewidth=0.04, linestyle=dashed, dash=0.17638889cm 0.10583334cm, arrowsize=0.05291666666666667cm 2.0,arrowlength=1.4,arrowinset=0.0]{->}(0.65555555,0.25597227)(0.65555555,1.1559722)
\psline[linecolor=black, linewidth=0.04, linestyle=dashed, dash=0.17638889cm 0.10583334cm, arrowsize=0.05291666666666667cm 2.0,arrowlength=1.4,arrowinset=0.0]{->}(0.65555555,-0.14402772)(0.65555555,-0.9440277)
\psline[linecolor=black, linewidth=0.04, linestyle=dashed, dash=0.17638889cm 0.10583334cm, arrowsize=0.05291666666666667cm 2.0,arrowlength=1.4,arrowinset=0.0]{->}(8.955556,0.15597227)(8.955556,1.1559722)
\psline[linecolor=black, linewidth=0.04, linestyle=dashed, dash=0.17638889cm 0.10583334cm, arrowsize=0.05291666666666667cm 2.0,arrowlength=1.4,arrowinset=0.0]{->}(8.955556,-0.14402772)(8.955556,-1.0440277)
\rput[bl](9.0,-0.044027727){$w$}
\rput[bl](0.65555555,-0.044027727){$w$}
\rput[bl](4.688889,0.48930562){$\chi_w$}
\rput[bl](1.75,-0.55){$T_1$}
\rput[bl](2.1888888,0.3){$0$}
\rput[bl](7.366667,0.3){$0$}
\rput[bl](7.6,-.55){$T_2$}
\psline[linecolor=black, linewidth=0.04](9.055555,1.3559723)(10.055555,1.3559723)
\psline[linecolor=black, linewidth=0.04](0.6666667,1.3559723)(0.0,1.3559723)
\rput[bl](5.6,-.85){$0$}
\rput[bl](4,-.85){$0$}
\rput[bl](9.2,-.85){$1$}
\rput[bl](0.2888889,-.85){$1$}
\rput[bl](9.388889,1.5448612){$0$}
\rput[bl](0.3888889,1.4893056){$0$}
\end{pspicture}
}}{.75}
\caption{Illustration of layer profile $\chi_w$: wall and obstacle
repulsion are embedded to the model with a fixed cost $W$ defined
in terms of the layer profile $\chi_w(x)\in[0,1]$, see
\eqref{e:wall}, which indicates proximity to walls less than a
width $w$, but vanishes away from obstacles and near exits to
allow a proper outflow of pedestrians. $\chi_w$ equals one at walls and is the linear function given by vertex values within triangles $T_1$, $T_2$ near exits.}
\label{fig-layerprofile}
\end{figure}
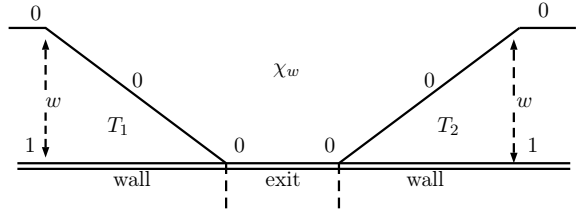

Finally all terms are coupled to the continuity equation with
velocity field $v(x,t)= -f(\rho) \mathcal{P}[\nabla \varphi]$ as
in the original model. At exits, we prescribe a maximum outflow,
given by $v(\xi,t)=-f(\rho) \vec n $ for all $ \xi\in\exit$.
Taking all these considerations into account the full macroscopic
model reads as:
\begin{subequations}
\begin{gather}
\partial_t \rho(x,t) + \nabla_x \cdot \left(- f(\rho(x,t)) \mathcal{P}[\nabla \varphi(x,t)] \rho(x,t) \right)  =0 \\
\varphi(x,t) = \frac{(\rho u \star \mathcal{K})(x,t)}{(\rho\star\mathcal{K})(x,t)} \label{e:localinteraction}\\
 u(x,t) =  \frac{\nabla_y \phi_{k^\text{opt}}(x,x,t)}{\norm{\nabla_y \phi_{k^\text{opt}}(x,x,t)}} (\phi_{k^{\text{opt}+1}}(x,x,t)-\phi_{k^\text{opt}}(x,x,t)) \label{e:conviction}\\
 k^{\text{opt}}(x,t) = \argmin\limits_k \phi_k(x,x,t) \\
 k^{\text{opt+1}}(x,t) = \argmin\limits_{k\neq k^{\text{opt}}} \phi_k(x,x,t)\label{e:secondbest} \\
\norm{\nabla_y \phi_k(x,y,t )} = \displaystyle{\begin{cases}
 \frac{1}{f(\rho(y,t)}  + W(y) & y\in V_x \\[1mm]
 \frac{1}{f(\rho_H)} & y\in H_x \\
 \end{cases}}\label{e:localeikonal}\\
 \text{s.t.:} \,\left.\phi_k\right\vert_{\partial \Omega_{E_k}} = 0 , k=1,\dots,M \, , \forall t \\
  \text{s.t.:} \,\mathcal{P}[\varphi(\xi,t)] = \vec n \,,\,\, \forall \xi\in\exit\,,\,\,\, \label{e:maxoutflow}
  \text{s.t.:} \,\rho(x,0)=\rho_0(x).
\end{gather}\label{e:localmodel}
\end{subequations}
We conclude the section with remarks on specific modeling assumptions.
\begin{remark}[vision cones]
We have set aside formal statements regarding assumptions on the visible set $V_x$, but clearly we think of at least regular, connected and closed sets.
A necessary condition is
\begin{equation}\label{e:Vinfo}
x \in  V^\circ_x = \operatorname{int} V_x,
\end{equation}
which implies that every pedestrian perceives some information
from all directions. This restriction rules out e.g. angular
vision cones (see Degond et.al.\cite{DRMPT2013}) where pedestrians
do not see what is happening behind them. In our model,
\eqref{e:Vinfo} is necessary to exclude unrealistic situations
where the chosen walking direction points outside the visible
area. The inclusion of angular-dependent vision cones is certainly
possible, but would imply a velocity-dependency and lead towards a
second-order macroscopic model.
\end{remark}

\begin{remark}[conviction term]
The introduction of the conviction term $u(x)$ requires the
computation of exit costs $\phi_k$ via the eikonal equation for
individual exits, which appears to be a significant complication
of the model. However, it is worth noting that the mechanism is
almost identical to the original model. In equation
\eqref{e:hugheseikonal} the costs of walking towards any of the
$K$ exits are compared, but only the minimal costs are used. Here,
we simply store more information. This connection is also
illustrated by looking at the numerical schemes for solving the
eikonal equation (see also Section ?): If a Fast Sweeping Method
is used in e.g. a corridor with two exists, this essentially
corresponds to solving for each exit separately if the
minimization step is left out. If a Fast Marching Method is used,
the conviction is directly related to the sequence in which
vertices are promoted, with the least convinced vertex being
assigned a cost the latest.
\end{remark}

\begin{remark}[waiting behavior]
The relaxation $\norm{\mathcal{P}[x]}\leq 1$ implies that the
modulus of the flux can be less than $f(\rho)\rho$ when
pedestrians are undecided. This makes a rigorous analysis of the
model equations a difficult task, which is not tackled in this
work. The benefit of our formulation is that the problem of
discontinuous velocity fields at sonic points has disappeared.
Pedestrians at those hyper surfaces will not move unless the sonic
points move.
\end{remark}

\subsection{The one-dimensional case}
Consider a one-dimensional corridor $\Omega=[0,1]$ with two exits
and the uniform radial vision cone $V_x=[x-L/2,x+L/2]\cap \Omega$
of length $L$. Exit costs towards the left and right exit are
computed at $y\in V_x$ as
\begin{gather*}
\phi_L(x,y,t) = \int\limits_{z<y, z\in H_x} \frac{1}{f(\rho_H)} \dd z + \int\limits_{z<y, z\in V_x} \frac{1}{f(\rho(z,t))} \dd z,  \\
\phi_R(x,y,t) = \int\limits_{z>y, z\in H_x} \frac{1}{f(\rho_H)} \dd z + \int\limits_{z>y, z\in V_x} \frac{1}{f(\rho(z,t))} \dd z , \\
u(x,t)=\phi_L(x,x,t)-\phi_R(x,x,t),
\end{gather*}
as illustrated in Fig. \ref{fig-1dillustration}.
\begin{figure}
\centering \psxinput{\begin{pspicture}(0,-2.8291457)(12.033044,2.8291457)
\definecolor{colour0}{rgb}{0.8,0.8,1.0}
\psline[linecolor=black, linewidth=0.04](0.31304348,-1.550885)(11.913043,-1.550885)
\psline[linecolor=black, linewidth=0.04](0.31304348,-1.3508849)(0.31304348,-1.7508849)
\psline[linecolor=black, linewidth=0.04](11.913043,-1.3508849)(11.913043,-1.7508849)
\psline[linecolor=black, linewidth=0.04](5.9130435,-1.3508849)(5.9130435,-1.7508849)
\rput[bl](0.31304348,-2.1508849){$0$}
\rput[bl](5.9130435,-2.1508849){$x_0$}
\rput[bl](11.913043,-2.027808){$1$}
\psframe[linecolor=black, linewidth=0.04, fillstyle=solid,fillcolor=colour0, dimen=outer](11.913043,-0.8030588)(8.869565,-1.5856675)
\psframe[linecolor=black, linewidth=0.04, fillstyle=solid,fillcolor=colour0, dimen=outer](3.652174,-0.055232737)(1.5652174,-1.5856675)
\psline[linecolor=black, linewidth=0.04, arrowsize=0.05291666666666667cm 2.0,arrowlength=1.4,arrowinset=0.0]{<->}(2.3347826,-2.3943632)(9.639131,-2.3943632)
\rput[bl](5.7130437,-2.8291457){$V_{x_0}$}
\psline[linecolor=black, linewidth=0.04, linestyle=dashed, dash=0.17638889cm 0.10583334cm](2.4130435,-2.1508849)(2.4347825,2.775202)
\psline[linecolor=black, linewidth=0.04, linestyle=dashed, dash=0.17638889cm 0.10583334cm](9.534782,-2.1552327)(9.534782,2.775202)
\psframe[linecolor=black, linewidth=0.04, fillstyle=solid,fillcolor=colour0, dimen=outer](8.347826,1.6143324)(7.652174,-1.5856675)
\psline[linecolor=black, linewidth=0.04](0.3478261,-1.5856675)(0.3478261,2.7621586)(0.3478261,2.7621586)
\psline[linecolor=black, linewidth=0.04](0.3478261,-1.5856675)(2.4347825,-0.8030588)(3.652174,0.6752021)(5.9130435,1.3708543)
\psline[linecolor=black, linewidth=0.04](11.913043,-1.5856675)(9.565217,-0.7161023)(8.869565,-0.10740665)(8.347826,0.24041943)(7.478261,2.153463)(5.9130435,2.6752021)(5.9130435,2.6752021)
\psline[linecolor=black, linewidth=0.04, linestyle=dotted, dotsep=0.10583334cm](5.9130435,1.3708543)(7.2173915,1.8056369)
\psline[linecolor=black, linewidth=0.04, linestyle=dotted, dotsep=0.10583334cm](5.826087,2.6752021)(5.478261,2.8491151)
\rput[bl](4.552174,1.1099846){$\phi_L$}
\rput[bl](7.517391,2.201289){$\phi_R$}
\rput[bl](4.8391304,1.8665063){$|u(x_0)|$}
\psline[linecolor=black, linewidth=0.04, linestyle=dashed, dash=0.17638889cm 0.10583334cm, arrowsize=0.05291666666666667cm 2.0,arrowlength=1.4,arrowinset=0.0]{<->}(5.9130435,1.4578108)(5.9130435,2.5882456)
\psframe[linecolor=black, linewidth=0.04, fillstyle=solid,fillcolor=colour0, dimen=outer](6.12357,-1.4140428)(5.702517,-1.5719376)
\end{pspicture}}{.75} \caption{Illustration of
path optimisation mechanism in 1D: A pedestrian located at $x_0$
computes and compares the cost potential $\phi_L,\phi_R$ of left
vs.\@ right exit in a corridor $[0,1]$. Next to its own negligible
density, the present crowd consists of three blocks. Outside the
vision cone $V_x$, the evacuation costs grow linearly at constant
rate, as the local density is unknown. Within $V_x$, the slope of
the cost potential increases with the pedestrian density.
Preference is then given towards the exit with lower estimated
cost. The conviction towards this decision is given as the cost
benefit $|u(x_0)|=|\phi_L-\phi_R|$.} \label{fig-1dillustration}
\end{figure}
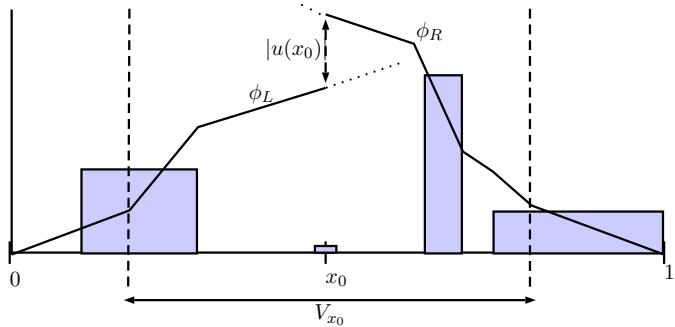
The cost potential $\phi$ is two-dimensional and $\partial_y
\phi(x,y)$ gives the preferred walking direction that a pedestrian
located at $x$ seeing $V_x$ assigns to $y\in[0,1]$. The walking
directions chosen prior to the consensus process are hence given
as $\partial_y \phi(x,x)$ along the diagonal of $[0,1]^2$. For
every fixed $x\in[0,1]$, there is a unique sonic point $z(x)$,
where $\phi_L(x,z(x))=\phi_R(x,z(x))$ and $\partial_y
\phi(x,z(x))$ does not exist. As illustrated in Fig.
\ref{f:1Ddirectionswitching}, the individually preferred walking
directions can switch multiple times between both exits, depending
on the current density and the vision cones. At switching points,
the preferred directions can point outwards (separation) or
inwards (collision) and only the weighted interaction process
\eqref{e:localinteraction}-\eqref{e:conviction} generates a smooth
velocity profile. In the Hughes' model, all vision cones are
identical and there is a single separation point.
\begin{figure}
\centering
\begin{tabular}{c c}
\psxinput{\input{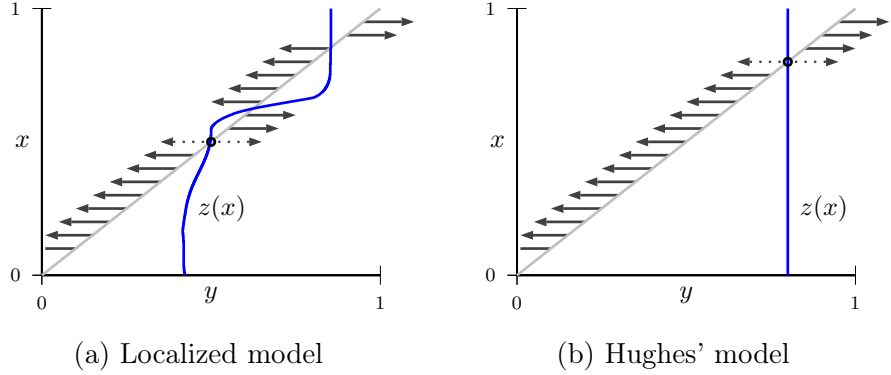}}{.9}&
\psxinput{\psset{xunit=5cm,yunit=3.9435cm}
\begin{pspicture*}(-.2,-.2)(1.1,1.1)
\psaxes[labelFontSize=\scriptstyle,Ox=0,Oy=0,Dx=1,Dy=1]{-}(0,0)(0,0)(1,1)
\psline[linecolor=darkgray, linewidth=0.04, arrowsize=0.053cm 2.0,arrowlength=1.4,arrowinset=0.0]{-}(0.1,0.1)(0.01,0.1)
\psline[linecolor=darkgray, linewidth=0.04, arrowsize=0.053cm 2.0,arrowlength=1.4,arrowinset=0.0]{->}(0.15,0.15)(0.01,0.15)
\psline[linecolor=darkgray, linewidth=0.04, arrowsize=0.053cm 2.0,arrowlength=1.4,arrowinset=0.0]{->}(0.2,0.2)(0.05,0.2)
\psline[linecolor=darkgray, linewidth=0.04, arrowsize=0.053cm 2.0,arrowlength=1.4,arrowinset=0.0]{->}(0.25,0.25)(0.1,0.25)
\psline[linecolor=darkgray, linewidth=0.04, arrowsize=0.053cm 2.0,arrowlength=1.4,arrowinset=0.0]{->}(0.3,0.3)(0.15,0.3)
\psline[linecolor=darkgray, linewidth=0.04, arrowsize=0.053cm 2.0,arrowlength=1.4,arrowinset=0.0]{->}(0.35,0.35)(0.2,0.35)
\psline[linecolor=darkgray, linewidth=0.04, arrowsize=0.053cm 2.0,arrowlength=1.4,arrowinset=0.0]{->}(0.4,0.4)(0.25,0.4)
\psline[linecolor=darkgray, linewidth=0.04, arrowsize=0.053cm 2.0,arrowlength=1.4,arrowinset=0.0]{->}(0.45,0.45)(0.3,0.45)
\psline[linecolor=darkgray, linewidth=0.04, arrowsize=0.053cm 2.0,arrowlength=1.4,arrowinset=0.0]{->}(0.5,0.5)(0.35,0.5)
\psline[linecolor=darkgray, linewidth=0.04, arrowsize=0.053cm 2.0,arrowlength=1.4,arrowinset=0.0]{->}(0.55,0.55)(0.4,0.55)
\psline[linecolor=darkgray, linewidth=0.04, arrowsize=0.053cm 2.0,arrowlength=1.4,arrowinset=0.0]{->}(0.6,0.6)(0.45,0.6)
\psline[linecolor=darkgray, linewidth=0.04, arrowsize=0.053cm 2.0,arrowlength=1.4,arrowinset=0.0]{->}(0.65,0.65)(0.5,0.65)
\psline[linecolor=darkgray, linewidth=0.04, arrowsize=0.053cm 2.0,arrowlength=1.4,arrowinset=0.0]{->}(0.7,0.7)(0.55,0.7)
\psline[linecolor=darkgray, linewidth=0.04, arrowsize=0.053cm 2.0,arrowlength=1.4,arrowinset=0.0]{->}(0.75,0.75)(0.6,0.75)

\psline[linecolor=darkgray, linewidth=0.04, arrowsize=0.053cm 2.0,arrowlength=1.4,arrowinset=0.0,linestyle=dotted]{->}(0.8,0.8)(0.65,0.8)
\psline[linecolor=darkgray, linewidth=0.04, arrowsize=0.053cm 2.0,arrowlength=1.4,arrowinset=0.0,linestyle=dotted]{->}(0.8,0.8)(0.95,0.8)
\psline[linecolor=darkgray, linewidth=0.04, arrowsize=0.053cm 2.0,arrowlength=1.4,arrowinset=0.0]{->}(0.85,0.85)(1,0.85)
\psline[linecolor=darkgray, linewidth=0.04, arrowsize=0.053cm 2.0,arrowlength=1.4,arrowinset=0.0]{->}(0.9,0.9)(1.05,.9)
\psline[linecolor=darkgray, linewidth=0.04, arrowsize=0.053cm 2.0,arrowlength=1.4,arrowinset=0.0]{->}(0.95,0.95)(1.1,.95)

\uput[-90](0.5,0){$y$}
\uput[180](0,0.5){$x$}
\uput[0](0.8,0.25){$z(x)$}

\psline[linecolor=lightgray, linewidth=0.04]{-}(0,0)(1,1)

\psline[linecolor=lightgray, linewidth=0.04, linecolor=blue]{-}(0.8,0)(0.8,1)

\pscircle[linecolor=black, linewidth=0.04, dimen=outer](0.8,0.8){0.075}

\end{pspicture*}}{.9}\\
(a)  Localized model &
 (b) Hughes' model
 \end{tabular}
\caption{Illustration of turning decisions of a 1D population for
a given $\rho(x,t)$: For every point $x,$ we show the individual
sonic point $z(x)$, where the costs $
\phi_L(x,z(x))=\phi_R(x,z(x))$ coincide. The preferred walking
direction for a pedestrian at $x$ is found along the diagonal
$(x,x)$. If the sonic point is to the left, the pedestrians aim to
walk towards the right and vice-versa. No direction is preferred
and the conviction is zero if the curve of sonic points intersects
the diagonal. (a) Local vision cones: The preferred direction
alters and creates multiple points of separation and collision.
The resulting velocity field is obtained by the smoothening
interaction process
\eqref{e:localinteraction}-\eqref{e:conviction} . (b) Hughes'
model: All vision cones coincide, hence there is one identical
sonic point common to all pedestrians.}
\label{f:1Ddirectionswitching}
\end{figure}

\subsection{Microscopic interpretation}

We conclude this section by briefly commenting on the modelling of
local vision at the microscopic level. The microscopic
modification is straightforward and uses the same ideas as at the
macroscopic level. It corresponds to updating the position $X =
X(t)$ according to a potential which depends on local information
only. Its calculation is based on the same equations as in the
macroscopic model \eqref{e:localeikonal} but using the smoothed
empirical density $\rho^N_g$ instead of $\rho$. The position
update is based on equations
\eqref{e:localinteraction}-\eqref{e:secondbest}. Hence individuals
choose the path towards the exit with the lowest cost, but weigh
their decision according to the predominant direction chosen
around them. For further details on the implementation we refer to
Algorithm \ref{a:microsim} presented in Section
\ref{s:computmeth}.

\subsection{Analysis of the domains of dependence}

In this subsection, we will discuss some mathematical properties
of the solutions of the eikonal equations \eqref{e:localeikonal}.
From the construction of the model, the potential $\phi(x,y,t)$
has to be computed for every $x\in \Omega$ on the entire domain
$\Omega$, which counterbalances the idea of locality and increases
the computational cost considerably. We show here, that the
computation of the potential can actually be reduced to a subset
of $\Omega$, called the effective domain of dependence, for every
$x$. Only this subset, which contains $V_x$, is considered in the
individual local planning problem and corresponds to the reduction
of the computational cost.

The following proofs rely crucially on the optimal path property of
the characteristics associated to the eikonal equation
\eqref{e:localeikonal}. We recall\cite{Holm,BC} that by Fermat's
principle the characteristic paths associated to $\phi(x,y,t)$,
given by the solution of:
\begin{equation}\label{e:chareikonal}
\gamma_{x,t}^z(s) \subset \Omega: \gamma(0)=z, \dot{\gamma}(s) =-
\nabla \phi(x,\gamma(s),t) \,\, \mbox{ for all} \, s\geq 0\,,
\end{equation}
are the optimal paths for the cost defined as
$$
c(y,t)=\displaystyle{\begin{cases}
 \frac{1}{f(\rho(y,t))}  + W(y) & y\in V_x \\[1mm]
 \frac{1}{f(\rho_H)} & y\in H_x \\
 \end{cases}}\,.
$$
Moreover, the potential is the value function for that cost. Hence
it is decreasing along these paths and satisfies the optimality
condition
\begin{equation}\label{e:optimality}
\phi(x,\gamma_{x,t}^z(a),t)-\phi(x,\gamma_{x,t}^z(b),t)= \int_a^b
c(\gamma_{x,t}^z(s),t) \,ds\,, \mbox{ for all } 0\leq a<b\,,
\end{equation}
being zero at its corresponding exit $\partial
\Omega_{\text{exit}}$. Furthermore, the curves $\gamma_{x,t}^z$
are the optimal paths to achieve the exit, i.e., they verify the
following global optimality condition
\begin{equation}\label{e:optimality2}
+\phi(x,z,t)= \int_0^{T_z} c(\gamma_{x,t}^z(s),t) \,ds \leq
\int_0^{\tilde T_z} c(\tilde\gamma(s),t) \,ds \,,
\end{equation}
for all $\tilde\gamma$ curves joining $z$ to any point in the exit
$\partial \Omega_{\text{exit}}$, where $T_z$ is the optimal time
to achieve the exit for the point $z\in\Omega$ and $\tilde T_z$ is
the time to achieve the exit for the path $\tilde\gamma$.

\begin{lemma}
\label{lemma-Mreduction} Consider any fixed $V_x\subset\Omega$ and
that $f(\rho)>0$, $0\leq\rho<\rho_{\max}$. Let $\phi_H$ be the
global solution of the eikonal equation $\norm{\nabla \phi_H } =
1/f(\rho_H), \phi_H(x)=0 \text{ on } \partial
\Omega_{\text{exit}}$. Define the minimum of $\phi_H$ in $V_x$ as
\begin{equation*}
m_H:= \min_{z\in V_x} \phi_H(z)
\end{equation*}
and the corresponding superlevel set of $\phi_H$ as
\begin{equation*}
M_H := \{x\in \Omega: \phi_H(x) \geq m_\phi \}.
\end{equation*}
Then the problem of computing the local potential $\phi(x,y,t)=:
\tilde{\phi}(y)$ out of \eqref{e:localeikonal} on $\Omega$ reduces
to the following problem on $M_\phi$:
\begin{equation*}
\begin{cases} \norm{\nabla_y \tilde{\phi}(y) }= \frac{1}{f(\rho(y,t))}+W(y) &\text{ in } V_x \\
\norm{\nabla_y \tilde{\phi}(y) }= \frac{1}{f(\rho_H)} &\text{ in } M_H \backslash V_x \\
\tilde{\phi}(y) = m_H &\text{ on } \partial M_H \backslash
\wall\,\, (\text{B.C. } )
\end{cases} \,.
\end{equation*}
\end{lemma}

\begin{proof}
If an exit is visible then $m_H=0, M_H=\Omega$ and the assertion
is trivial. If no exit is visible then by construction $V_x
\subset M_H$ and $\phi_H= m_H>0$ on $\partial M_H$. As the walking
costs are always positive, $c(y,t)>0$, we get $\phi(x,y,t)>m_H$
for all $y \in \operatorname{int} M_H$. On the other hand, any
point $z\in\Omega\backslash M_H$ satisfies $\phi(x,z,t)<m_H$ and
hence $\gamma_{x,t}^z(s)$ does not intersect $V_x$, otherwise the
cost should be larger at a middle point than initially
contradicting the optimality of the path $\gamma_{x,t}^z(s)$ in
\eqref{e:optimality}. Hence $\partial M_H$ is the maximal level
set consisting of points whose optimal paths do not cross $V_x$,
and therefore, $\phi(x,z,t)$ can be computed from
\eqref{e:localeikonal} with constant right-hand side outside
$M_H$.
\end{proof}

\begin{definition}
Consider a fixed visibility area $V_x$. For a $z\in\Omega$, denote the
\emph{default optimal path} $\gamma_H^z$ as the parameterised
curve associated to a gradient walk along $\phi_H$ starting in
$z$, that is
\begin{equation*}
\gamma_H^z(s) \subset \Omega: \gamma(0)=z, \dot{\gamma}(s) =-
\nabla \phi_H(\gamma(s)) \,\, \forall \, s\geq 0.
\end{equation*}
Next, define the \emph{characteristics' shadow} $V^\#$ as the set of all points, whose default optimal paths crosses the visibility area, hence
\begin{equation*}
V^\# := \{ z\in\Omega: \gamma_H^z \cap \operatorname{int} V_x \neq
\emptyset \}.
\end{equation*}
\end{definition}

Note, that  $V^\# \subset M_H$ since any default optimal path
outside of $M_H$ cannot intersect with $V_x$ as proven in the
previous lemma.

\begin{lemma}\label{lemma-Vreduction}
Consider any fixed $V_x\subset\Omega$ and assume that $f(\rho)>0$ is
increasing in $0\leq\rho<\rho_{\max}$, with $\rho_H=0$, then the
problem of computing the local potential $\tilde{\phi}(y)$ out of
\eqref{model-local-eikonal} further reduces to the following
problem on $V^\#$
\begin{equation*}
\begin{cases} \norm{\nabla_y \tilde{\phi}(y) }= \frac{1}{f(\rho(y,t))}+W(y) &\text{ in } V_x \\
\norm{\nabla_y \tilde{\phi}(y) }= \frac{1}{f(0)} &\text{ in }  V^\# \backslash V_x\\
\tilde{\phi} \equiv \phi_H &\text{ on } \partial V^\#
\end{cases}
\,.
\end{equation*}
\end{lemma}

\begin{proof} For any point $z$ whose default optimal path
$\gamma_H^z$ that does not intersect with $V$, the claim is that
$\tilde{\phi}(z)=\phi_H(z)$ due the monotonicity of the cost
function, i.e.,
$$
\frac{1}{f(\rho(y,t))}+W(y) \geq \frac{1}{f(0)}\,.
$$

To prove this, let us denote by $T_z^H$ the optimal time to get to
the exit for the default optimal path $\gamma_H^z$.

We first take $\gamma_{x,t}^z(s)$ as a candidate path in the
global optimality condition \eqref{e:optimality2} for the eikonal
equation with right hand side $c_H=\tfrac{1}{f(0)}$. Being
$\gamma_{x,t}^z(s)$ a path joining $z$ to a point in the exit and
$\gamma_H^z(s)$ the optimal one, we conclude
$$
\phi_H(z) = T_z^H c_H \leq T_z c_H \leq \int_0^{T_z}
c(\gamma_{x,t}^z(s),t) \,ds = \tilde\phi(z) \,.
$$
Now, we take $\gamma_H(z)$ as a candidate path in the global
optimality condition \eqref{e:optimality2} for the eikonal
equation with right hand side $c(y,t)$. It is an admissible path
as it connects $z$ to a point at the exit and the cost along its path
coincides with $c_H=c(\gamma_H^z(s),t)$ for all $s\in [0,T_Z^H]$
since the path does not cross $V$. Then, we get
$$
\tilde\phi(z) \leq \int_0^{T_z^H} c(\gamma_H^z(s),t) \,ds = T_z^H
c_H = \phi_H(z) \,,
$$
leading to the stated result.
\end{proof}

\begin{figure}
\begin{center}
\psxinput{
{
\begin{pspicture}(0,-1.62)(12.44,1.62)
\psline[linecolor=black, linewidth=0.04](1.62,1.6)(1.62,-1.6)
\psline[linecolor=black, linewidth=0.04](4.22,1.6)(4.22,-1.6)
\psline[linecolor=black, linewidth=0.04, arrowsize=0.05291666666666667cm 2.0,arrowlength=1.4,arrowinset=0.0]{<->}(1.82,-0.8)(4.02,-0.8)
\psline[linecolor=black, linewidth=0.04, linestyle=dotted, dotsep=0.10583334cm, arrowsize=0.05291666666666667cm 2.0,arrowlength=1.4,arrowinset=0.0]{<->}(1.82,-1.0)(6.22,-1.0)
\psline[linecolor=black, linewidth=0.04, linestyle=dashed, dash=0.17638889cm 0.10583334cm, arrowsize=0.05291666666666667cm 2.0,arrowlength=1.4,arrowinset=0.0]{<->}(1.82,-1.2)(10.42,-1.2)
\psline[linecolor=black, linewidth=0.04, linestyle=dashed, dash=0.17638889cm 0.10583334cm](10.62,1.6)(10.62,-1.6)
\psline[linecolor=black, linewidth=0.04, arrowsize=0.05291666666666667cm 2.0,arrowlength=1.4,arrowinset=0.0]{->}(11.02,1.2)(12.02,1.2)
\psline[linecolor=black, linewidth=0.04, arrowsize=0.05291666666666667cm 2.0,arrowlength=1.4,arrowinset=0.0]{->}(11.02,0.4)(12.02,0.4)
\psline[linecolor=black, linewidth=0.04, arrowsize=0.05291666666666667cm 2.0,arrowlength=1.4,arrowinset=0.0]{->}(11.02,-1.2)(12.02,-1.2)
\psline[linecolor=black, linewidth=0.04, arrowsize=0.05291666666666667cm 2.0,arrowlength=1.4,arrowinset=0.0]{->}(1.22,1.2)(0.22,1.2)
\psline[linecolor=black, linewidth=0.04, arrowsize=0.05291666666666667cm 2.0,arrowlength=1.4,arrowinset=0.0]{->}(1.22,0.4)(0.22,0.4)
\psline[linecolor=black, linewidth=0.04, arrowsize=0.05291666666666667cm 2.0,arrowlength=1.4,arrowinset=0.0]{->}(1.22,-0.4)(0.22,-0.4)
\psline[linecolor=black, linewidth=0.04, arrowsize=0.05291666666666667cm 2.0,arrowlength=1.4,arrowinset=0.0]{->}(1.22,-1.2)(0.22,-1.2)
\psline[linecolor=black, linewidth=0.04, arrowsize=0.05291666666666667cm 2.0,arrowlength=1.4,arrowinset=0.0]{->}(11.02,-0.4)(12.02,-0.4)
\rput[bl](2.82,-0.6){$V$}
\rput[bl](5.02,-0.85){$V^\#$}
\rput[bl](7.42,-1.1){$M_{H}$}
\psline[linecolor=black, linewidth=0.04, linestyle=dotted, dotsep=0.10583334cm](6.42,1.6)(6.42,-1.6)
\psline[linecolor=black, linewidth=0.04](0.02,1.6)(0.02,-1.6)(12.42,-1.6)
\psline[linecolor=black, linewidth=0.04](12.42,-1.6)(12.42,1.6)(0.02,1.6)
\psline[linecolor=black, linewidth=0.04, linestyle=dotted, dotsep=0.10583334cm, arrowsize=0.05291666666666667cm 2.0,arrowlength=1.4,arrowinset=0.0]{->}(6.82,1.2)(10.42,1.2)
\psline[linecolor=black, linewidth=0.04, linestyle=dotted, dotsep=0.10583334cm, arrowsize=0.05291666666666667cm 2.0,arrowlength=1.4,arrowinset=0.0]{->}(6.82,0.4)(10.42,0.4)
\psline[linecolor=black, linewidth=0.04, linestyle=dotted, dotsep=0.10583334cm, arrowsize=0.05291666666666667cm 2.0,arrowlength=1.4,arrowinset=0.0]{->}(6.82,-0.4)(10.42,-0.4)
\end{pspicture}
}}{1} \caption{Illustration of the
domains $V,M_H$ and $V^{\#}$  for computation of the visibility
area potential $\tilde{\phi}$ for the case of a corridor with two
opposing exits: The problem on $\Omega$ generally reduces to a
HJ-equation on $M_H$, as by construction $\tilde{\phi}$ coincides
with $\phi_H$ outside of $M_H$ (Lemma \eqref{lemma-Mreduction},
$-\nabla \phi_H$ solid arrows). If $\rho_H=0$, any default optimal
path of $\phi_H$ that does not intersect $V$ remains optimal, as
indicated by dotted arrows, and the problem reduces to $V^\#$
(Lemma \ref{lemma-Vreduction}).} \label{fig-MVreduction}
\end{center}
\end{figure}
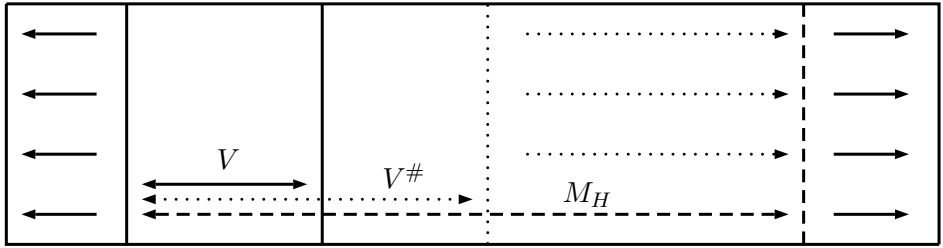
We illustrate Lemmata \ref{lemma-Mreduction} and \ref{lemma-Vreduction} in Figure \ref{fig-MVreduction}.
It can be seen, that the reduction of the computational domain
from $M_H$ to $V^{\#}$ can be significant, as the size of $M_H$
depends on the closeness of $V$ to the nearest exit, not on the
size of $V$. For the exemplary geometry of Figure \ref{fig-MVreduction}, the boundary of $V^{\#}$ coincides with
the sonic points of $\phi_H$, but this is not true in general.
Furthermore, it is easy to see why the computational domain cannot be reduced further. Suppose that $\rho(\cdot,t)$ is spatially homogeneous, then
$-\nabla\phi$ in
$V^\#\backslash V$
points to the left exit as in the eikonal case. On the other side, one can choose
a situation with a large density at the left boundary of
$V$ that leads to right-pointing $-\nabla\phi$ in
$V^\#\backslash V$.

\section{Computational methods}\label{s:computmeth}

In this section we present a microscopic and a macroscopic numerical solver to simulate the classic and the local
version of the Hughes model. The proposed methods have been implemented on regular and triangular meshes in 2D to allow for flexible
discretizations of polygonal domains with one or several obstacles.

For the macroscopic system \eqref{e:localmodel} we use the
following explicit iterative algorithm:
\begin{algorithm}
\caption{Macroscopic version of the localised model}
\label{a:macrosim}
Initialisation: \begin{itemize}
\item  A discretisation $\hat\Omega=( \mathcal{V},  \mathcal{E}, \mathcal{T})$ of $\Omega$ consisting of vertices, edges and cells.
\item An initial density $\hat\rho_0$ given on $\mathcal{T}$, such that $\int_{T} \rho(x,0)\dd x = \hat\rho_0(T) \,\, \forall \, T \in \mathcal{T}$.
\item A list of exits, a list of boundary edges per exit and $|\mathcal{V}|$ subsets of $\mathcal{V}$ containing the vision cones defined in terms of vertices.
\end{itemize}
\begin{enumerate}
\item Compute the cost potential $\hat\phi_k$ for all exits out of the current density $\hat\rho$ by solving \eqref{e:localeikonal} along the vertices for every $v\in \mathcal{V}$.
\item Determine the cell values of $\hat\phi_k$ and $\nabla\hat\phi_k$ by an averaging / finite difference approximation of the values at neighbouring vertices, e.g. $\hat\phi_k(T)=\frac{1}{|\{v\in\partial T\}|}\sum_{v\in\partial T} \hat\phi_k(u)$, and obtain $\hat u(T)$ here from. 
\item Compute a numerical convolution of $\hat u$ with $\mathcal{K}$, which gives $\hat\varphi$ on the cells.
\item Update the density with a cell-based Finite Volume Method using the velocity field $-f(\hat\rho)\mathcal{P}[\hat\varphi]$ and a suitably chosen time step.
\end{enumerate}
\label{algo:macro}
\end{algorithm}

The discretisation is either a regular grid or an unstructured
regular triangular mesh to allow more complex geometries. For
solving the eikonal equations, one can chose between Fast Sweeping
Methods\cite{zhao05,QZZ2007} and Fast Marching
Methods\cite{kimmel1998computing,SethianRev1999}. The former is
based on a Gauss-Seidel iteration, which updates the solution by
passing through the computational domain in alternate pre-defined
sweeping directions. A rectangular grid provides a natural
ordering of all grid points. This ordering does not exist on an
unstructured grid and is replaced by a general ordering strategy
by introducing reference points, which is done once.  Then the
solution at each node is consecutively updated by running through
the ordered lists. Marching methods update vertices in a monotone
increasing order, where in every iteration a list of candidate
values is available by finite difference approximation from
previously approved values. The smallest value all of candidate
values is then promoted and assigned to its vertex.

As a Finite Volume Method we use the first-order monotone FORCE
scheme\cite{toro2009force,Toro2010}. Some postprocessing between
the steps of Algorithm \ref{algo:macro} is required:
outward-pointing components of $\nabla\hat\phi_k$ are removed
along the boundary, suitable values of $\nabla\hat\phi_k$ are
ensured at corners of $\Omega$, and the max outflow condition
\eqref{e:maxoutflow} is enforced at cells neighbouring exit edges.

The analogous algorithm used for the numerical simulation of the
microscopic model is:
\begin{algorithm}
\caption{Microscopic version of the localised model}
\label{a:microsim}
Let us consider a system of $N$ particles, which are initially located at positions $X^j(0) = X^j_0$. In every time step
$t^i = i \Delta t$ we update the particle position as follows:
\begin{enumerate}
\item Determine the empirical density at time $t^i$:
\begin{align*}
\rho^N_g(x,t^i)  = \frac{1}{N} \sum_{j=1}^N g(x-X^j(t^i)),
\end{align*}
where $g$ denotes a Gaussian.
\item Solve the eikonal equation to determine the weighted distance to each exit $\phi^k = \phi^k(x,t^i)$, $k=1, \ldots M$:
\begin{subequations}\label{e:microeikonal}
\begin{align}
 \norm{\nabla \phi^k(x,y,t^i)} &=
\begin{cases}
\frac{1}{ f(\rho^N_g(y,t^i))} +W(y) &\text{ if } y \in V_x\\
\frac{1}{f(\rho_H)} &\text{ otherwise.}
\end{cases}\\
\phi^k(x,t^i) &= 0
\end{align}
\end{subequations}
\item Update the position of each particle $X^j$ via:
\begin{align}
\dot{X^j}(t^i)  = -f^2(\rho^N_g(x,t^i) \cdot \nabla \varphi(x,t^i))),\label{e:microupdate}
\end{align}
where $\varphi(x,t^i)$ is determined by \eqref{e:localinteraction}.
\end{enumerate}
\end{algorithm}

\section{Results}\label{s:numexp}

In this section we illustrate the dynamics of the localised model
for crowd dynamics with examples in one and two dimensions. In all
simulations we consider an evacuation scenario of a corridor,
where a given initial distribution of people tries to leave the
rectangular domain through either one of the two exits as fast as
possible. We compare the evacuation time, i.e. the time at which
all individuals have left the domain, with respect to different
parameters, e.g. vision cones. In the case of a global vision cone
we obtain Hughes' type dynamics. As a flux law, we chose the LWR
function \begin{equation}
f(\rho)=\rho(1-\rho)
\label{fluxlaw},
\end{equation}
setting $\rho_{\max}=1$ throughout this section.

\subsection{1D corridor - macroscopic model} \label{s:1dcorrmac}
In our first example the domain $\Omega$ corresponds to the unit
interval $\Omega=[0,1]$ with two exits located at either end, i.e.
at $x=0$ and $x=1$. We consider an evacuation scenario in which
two groups, one of them being densly packed, want to leave through
either one of the exits:
\begin{equation}
\rho_0(x) = \begin{cases}
0.85 & 0\geq x \geq 0.3 \\
0 & 0.3< x < 0.6 \\
0.25 & 0.6\geq x \geq 1
\end{cases},
\label{e:initsymcorridor}
\end{equation}
and we set the width of the vision cone to $L=0.75$. The resulting
dynamics are illustrated at 4 time steps in Fig.
\ref{fig-1Dturnaround}. Within the left block, some pedestrians
decide to walk towards the right exit, as they are aware of the
high density on their left and account for a higher walking cost
compared to the relatively empty right hand side. After the
separation the right-moving part evolves as a rarefaction wave, as
known from the LWR model. As the distance between the wave and the
left-moving shock grows, the effects of the local vision cone
become apparent. At some point pedestrians moving to the right
do not see the high density at the left exit anymore and start to
turn around. Therefore the rarefaction wave splits again - one
part continues while the other one turns around and moves back to
the left exit. The turn-around occurs in several stages:
\begin{enumerate}
\item A new sonic point arises, where pedestrians are undecided between both exits. The walking direction is unchanged as the local consensus process \eqref{e:localinteraction} prevents an immediate switching.
\item When a critical mass of density and conviction opting for walking to the left, the velocity after consensus switches continuously and passing through zero. This creates a temporary collision point, as there a still pedestrians to the left of the sonic point which walk towards the right.
\item The density at the collision point increases, which causes pedestrians to the left of the collision point to turn around too, as a higher density is in their way, as it can bee seen in Fig. \ref{fig-1Dturnaround}(c).
\item Finally, all pedestrians to the left of  the initial sonic point have turned and walk towards the left ( Fig. \ref{fig-1Dturnaround}(d)).
\end{enumerate}

\begin{figure}
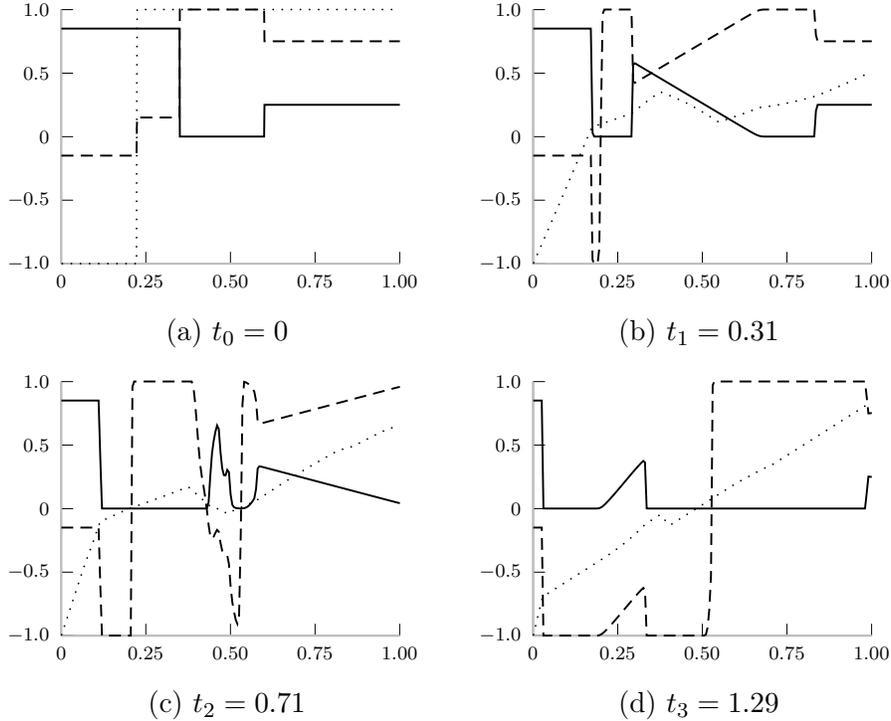

\centering
\begin{tabular}{c c}
\psxinput{\input{draw_num_1d_ex1_a}}{.9} &
\psxinput{\psset{xunit=5cm,yunit=1.8779cm}
\begin{pspicture*}(-0.16774,-1.3334)(1.1226,1.2433)
\psaxes[labelFontSize=\scriptstyle,Ox=0,Oy=-1.0,Dx=0.25,Dy=0.5,ticksize=5pt,linecolor=lightgray]{-}(0,-1.0)(0,-1)(1,1)
\savedata{\mydata}[{
{0.0005,-1}{0.0055,-0.96978}{0.0105,-0.9395}{0.0155,-0.90916}{0.0205,-0.87876}
{0.0255,-0.84829}{0.0305,-0.81778}{0.0355,-0.78721}{0.0405,-0.75658}{0.0455,-0.7259}{0.0505,-0.69518}
{0.0555,-0.6644}{0.0605,-0.63357}{0.0655,-0.6027}{0.0705,-0.57178}{0.0755,-0.54082}{0.0805,-0.50981}
{0.0855,-0.47876}{0.0905,-0.44766}{0.0955,-0.41653}{0.1005,-0.38536}{0.1055,-0.35415}{0.1105,-0.3229}
{0.1155,-0.29161}{0.1205,-0.26029}{0.1255,-0.22893}{0.1305,-0.19753}{0.1355,-0.1661}{0.1405,-0.13464}
{0.1455,-0.10315}{0.1505,-0.071619}{0.1555,-0.040061}{0.1605,-0.0084714}{0.1655,0.023148}{0.1705,0.054798}
{0.1755,0.07708}{0.1805,0.081408}{0.1855,0.085748}{0.1905,0.090115}{0.1955,0.09451}{0.2005,0.098931}
{0.2055,0.10338}{0.2105,0.10785}{0.2155,0.11235}{0.2205,0.11687}{0.2255,0.12141}{0.2305,0.12598}
{0.2355,0.13058}{0.2405,0.13519}{0.2455,0.13983}{0.2505,0.14449}{0.2555,0.14917}{0.2605,0.15387}
{0.2655,0.1586}{0.2705,0.16334}{0.2755,0.1681}{0.2805,0.17288}{0.2855,0.17768}{0.2905,0.18249}
{0.2955,0.18894}{0.3005,0.20041}{0.3055,0.21173}{0.3105,0.22286}{0.3155,0.23378}{0.3205,0.24452}
{0.3255,0.25507}{0.3305,0.26544}{0.3355,0.27564}{0.3405,0.28567}{0.3455,0.29554}{0.3505,0.30525}
{0.3555,0.31481}{0.3605,0.32423}{0.3655,0.3335}{0.3705,0.34263}{0.3755,0.35163}{0.3805,0.3468}
{0.3855,0.34185}{0.3905,0.33677}{0.3955,0.33157}{0.4005,0.32626}{0.4055,0.32083}{0.4105,0.3153}
{0.4155,0.30965}{0.4205,0.3039}{0.4255,0.29805}{0.4305,0.2921}{0.4355,0.28605}{0.4405,0.27991}
{0.4455,0.27368}{0.4505,0.26736}{0.4555,0.26094}{0.4605,0.2542}{0.4655,0.24684}{0.4705,0.23937}
{0.4755,0.23182}{0.4805,0.22419}{0.4855,0.21649}{0.4905,0.20871}{0.4955,0.20086}{0.5005,0.19294}
{0.5055,0.18495}{0.5105,0.17689}{0.5155,0.16876}{0.5205,0.16057}{0.5255,0.15231}{0.5305,0.14399}
{0.5355,0.1356}{0.5405,0.12716}{0.5455,0.11866}{0.5505,0.11343}{0.5555,0.1185}{0.5605,0.12352}
{0.5655,0.12849}{0.5705,0.1334}{0.5755,0.13826}{0.5805,0.14307}{0.5855,0.14783}{0.5905,0.15254}
{0.5955,0.15719}{0.6005,0.1618}{0.6055,0.16637}{0.6105,0.17088}{0.6155,0.17535}{0.6205,0.17978}
{0.6255,0.18432}{0.6305,0.18947}{0.6355,0.19457}{0.6405,0.19963}{0.6455,0.20466}{0.6505,0.20964}
{0.6555,0.21459}{0.6605,0.2195}{0.6655,0.22439}{0.6705,0.22874}{0.6755,0.23029}{0.6805,0.23187}
{0.6855,0.23355}{0.6905,0.23533}{0.6955,0.2372}{0.7005,0.23917}{0.7055,0.24123}{0.7105,0.24337}
{0.7155,0.2456}{0.7205,0.24791}{0.7255,0.25029}{0.7305,0.25276}{0.7355,0.25529}{0.7405,0.2579}
{0.7455,0.26058}{0.7505,0.26333}{0.7555,0.26614}{0.7605,0.26901}{0.7655,0.27195}{0.7705,0.27494}
{0.7755,0.278}{0.7805,0.28111}{0.7855,0.28427}{0.7905,0.28749}{0.7955,0.29076}{0.8005,0.29409}
{0.8055,0.29746}{0.8105,0.30088}{0.8155,0.30435}{0.8205,0.30787}{0.8255,0.31143}{0.8305,0.31503}
{0.8355,0.31905}{0.8405,0.32427}{0.8455,0.32961}{0.8505,0.33499}{0.8555,0.34041}{0.8605,0.34587}
{0.8655,0.35137}{0.8705,0.3569}{0.8755,0.36247}{0.8805,0.36807}{0.8855,0.37371}{0.8905,0.37938}
{0.8955,0.38509}{0.9005,0.39083}{0.9055,0.39659}{0.9105,0.40239}{0.9155,0.40822}{0.9205,0.41408}
{0.9255,0.41997}{0.9305,0.42589}{0.9355,0.43184}{0.9405,0.43781}{0.9455,0.44381}{0.9505,0.44984}
{0.9555,0.45589}{0.9605,0.46197}{0.9655,0.46807}{0.9705,0.4742}{0.9755,0.48036}{0.9805,0.48653}
{0.9855,0.49273}{0.9905,0.49895}{0.9955,0.5052},{1.0005,0.51021}
}]
\dataplot[plotstyle=line,linestyle=dotted, linecolor=black]{\mydata}
\savedata{\mydata}[{
{0.0005,-0.15}{0.0055,-0.15}{0.0105,-0.15}{0.0155,-0.15}{0.0205,-0.15}
{0.0255,-0.15}{0.0305,-0.15}{0.0355,-0.15}{0.0405,-0.15}{0.0455,-0.15}{0.0505,-0.15}
{0.0555,-0.15}{0.0605,-0.15}{0.0655,-0.15}{0.0705,-0.15}{0.0755,-0.15}{0.0805,-0.15}
{0.0855,-0.15}{0.0905,-0.15}{0.0955,-0.15}{0.1005,-0.15}{0.1055,-0.15}{0.1105,-0.15}
{0.1155,-0.15}{0.1205,-0.15}{0.1255,-0.15}{0.1305,-0.15}{0.1355,-0.15}{0.1405,-0.15}
{0.1455,-0.15}{0.1505,-0.15}{0.1555,-0.15}{0.1605,-0.15}{0.1655,-0.15}{0.1705,-0.15001}
{0.1755,-0.97041}{0.1805,-1}{0.1855,-1}{0.1905,-1}{0.1955,-0.88051}{0.2005,0.192}
{0.2055,0.9663}{0.2105,1}{0.2155,1}{0.2205,1}{0.2255,1}{0.2305,1}
{0.2355,1}{0.2405,1}{0.2455,1}{0.2505,1}{0.2555,1}{0.2605,1}
{0.2655,1}{0.2705,1}{0.2755,1}{0.2805,1}{0.2855,1}{0.2905,0.99994}
{0.2955,0.43601}{0.3005,0.42276}{0.3055,0.43065}{0.3105,0.43854}{0.3155,0.44644}{0.3205,0.45433}
{0.3255,0.46222}{0.3305,0.47011}{0.3355,0.478}{0.3405,0.48589}{0.3455,0.49378}{0.3505,0.50168}
{0.3555,0.50957}{0.3605,0.51746}{0.3655,0.52535}{0.3705,0.53325}{0.3755,0.54114}{0.3805,0.54903}
{0.3855,0.55692}{0.3905,0.56482}{0.3955,0.57271}{0.4005,0.5806}{0.4055,0.5885}{0.4105,0.59639}
{0.4155,0.60429}{0.4205,0.61218}{0.4255,0.62008}{0.4305,0.62797}{0.4355,0.63587}{0.4405,0.64377}
{0.4455,0.65166}{0.4505,0.65956}{0.4555,0.66746}{0.4605,0.67535}{0.4655,0.68325}{0.4705,0.69115}
{0.4755,0.69905}{0.4805,0.70695}{0.4855,0.71485}{0.4905,0.72275}{0.4955,0.73065}{0.5005,0.73855}
{0.5055,0.74645}{0.5105,0.75436}{0.5155,0.76226}{0.5205,0.77016}{0.5255,0.77806}{0.5305,0.78597}
{0.5355,0.79387}{0.5405,0.80177}{0.5455,0.80968}{0.5505,0.81758}{0.5555,0.82548}{0.5605,0.83339}
{0.5655,0.84129}{0.5705,0.84919}{0.5755,0.85709}{0.5805,0.86499}{0.5855,0.87288}{0.5905,0.88078}
{0.5955,0.88866}{0.6005,0.89654}{0.6055,0.90442}{0.6105,0.91228}{0.6155,0.92013}{0.6205,0.92796}
{0.6255,0.93576}{0.6305,0.94352}{0.6355,0.95123}{0.6405,0.95885}{0.6455,0.96634}{0.6505,0.97363}
{0.6555,0.98058}{0.6605,0.98697}{0.6655,0.99244}{0.6705,0.9965}{0.6755,0.99885}{0.6805,0.99977}
{0.6855,0.99997}{0.6905,1}{0.6955,1}{0.7005,1}{0.7055,1}{0.7105,1}
{0.7155,1}{0.7205,1}{0.7255,1}{0.7305,1}{0.7355,1}{0.7405,1}
{0.7455,1}{0.7505,1}{0.7555,1}{0.7605,1}{0.7655,1}{0.7705,1}
{0.7755,1}{0.7805,1}{0.7855,1}{0.7905,1}{0.7955,1}{0.8005,1}
{0.8055,1}{0.8105,1}{0.8155,1}{0.8205,1}{0.8255,1}{0.8305,0.9992}
{0.8355,0.79859}{0.8405,0.7502}{0.8455,0.75}{0.8505,0.75}{0.8555,0.75}{0.8605,0.75}
{0.8655,0.75}{0.8705,0.75}{0.8755,0.75}{0.8805,0.75}{0.8855,0.75}{0.8905,0.75}
{0.8955,0.75}{0.9005,0.75}{0.9055,0.75}{0.9105,0.75}{0.9155,0.75}{0.9205,0.75}
{0.9255,0.75}{0.9305,0.75}{0.9355,0.75}{0.9405,0.75}{0.9455,0.75}{0.9505,0.75}
{0.9555,0.75}{0.9605,0.75}{0.9655,0.75}{0.9705,0.75}{0.9755,0.75}{0.9805,0.75}
{0.9855,0.75}{0.9905,0.75}{0.9955,0.75},{1.0005,0.75}
}]
\dataplot[plotstyle=line,linestyle=dashed,linecolor=black]{\mydata}
\savedata{\mydata}[{
{0.0005,0.85}{0.0055,0.85}{0.0105,0.85}{0.0155,0.85}{0.0205,0.85}
{0.0255,0.85}{0.0305,0.85}{0.0355,0.85}{0.0405,0.85}{0.0455,0.85}{0.0505,0.85}
{0.0555,0.85}{0.0605,0.85}{0.0655,0.85}{0.0705,0.85}{0.0755,0.85}{0.0805,0.85}
{0.0855,0.85}{0.0905,0.85}{0.0955,0.85}{0.1005,0.85}{0.1055,0.85}{0.1105,0.85}
{0.1155,0.85}{0.1205,0.85}{0.1255,0.85}{0.1305,0.85}{0.1355,0.85}{0.1405,0.85}
{0.1455,0.85}{0.1505,0.85}{0.1555,0.85}{0.1605,0.85}{0.1655,0.85}{0.1705,0.84999}
{0.1755,0.029594}{0.1805,4.0819e-11}{0.1855,4.4336e-20}{0.1905,4.8317e-29}{0.1955,8.9481e-33}{0.2005,4.9231e-33}
{0.2055,1.088e-32}{0.2105,4.186e-32}{0.2155,1.6494e-31}{0.2205,6.4979e-31}{0.2255,2.5594e-30}{0.2305,1.0077e-29}
{0.2355,3.9657e-29}{0.2405,1.5596e-28}{0.2455,6.1278e-28}{0.2505,2.405e-27}{0.2555,9.4273e-27}{0.2605,3.6893e-26}
{0.2655,1.4409e-25}{0.2705,5.6143e-25}{0.2755,2.37e-24}{0.2805,1.4093e-18}{0.2855,9.3989e-12}{0.2905,5.625e-05}
{0.2955,0.56399}{0.3005,0.57724}{0.3055,0.56935}{0.3105,0.56146}{0.3155,0.55356}{0.3205,0.54567}
{0.3255,0.53778}{0.3305,0.52989}{0.3355,0.522}{0.3405,0.51411}{0.3455,0.50622}{0.3505,0.49832}
{0.3555,0.49043}{0.3605,0.48254}{0.3655,0.47465}{0.3705,0.46675}{0.3755,0.45886}{0.3805,0.45097}
{0.3855,0.44308}{0.3905,0.43518}{0.3955,0.42729}{0.4005,0.4194}{0.4055,0.4115}{0.4105,0.40361}
{0.4155,0.39571}{0.4205,0.38782}{0.4255,0.37992}{0.4305,0.37203}{0.4355,0.36413}{0.4405,0.35623}
{0.4455,0.34834}{0.4505,0.34044}{0.4555,0.33254}{0.4605,0.32465}{0.4655,0.31675}{0.4705,0.30885}
{0.4755,0.30095}{0.4805,0.29305}{0.4855,0.28515}{0.4905,0.27725}{0.4955,0.26935}{0.5005,0.26145}
{0.5055,0.25355}{0.5105,0.24564}{0.5155,0.23774}{0.5205,0.22984}{0.5255,0.22194}{0.5305,0.21403}
{0.5355,0.20613}{0.5405,0.19823}{0.5455,0.19032}{0.5505,0.18242}{0.5555,0.17452}{0.5605,0.16661}
{0.5655,0.15871}{0.5705,0.15081}{0.5755,0.14291}{0.5805,0.13501}{0.5855,0.12712}{0.5905,0.11922}
{0.5955,0.11134}{0.6005,0.10346}{0.6055,0.095582}{0.6105,0.087719}{0.6155,0.079871}{0.6205,0.072042}
{0.6255,0.064241}{0.6305,0.056478}{0.6355,0.048772}{0.6405,0.041151}{0.6455,0.033659}{0.6505,0.026371}
{0.6555,0.01942}{0.6605,0.01303}{0.6655,0.0075639}{0.6705,0.0034972}{0.6755,0.0011459}{0.6805,0.00023335}
{0.6855,2.6345e-05}{0.6905,1.4998e-06}{0.6955,3.8927e-08}{0.7005,4.0364e-10}{0.7055,1.3837e-12}{0.7105,1.167e-15}
{0.7155,1.4296e-19}{0.7205,7.3733e-25}{0.7255,3.2669e-82}{0.7305,2.1973e-78}{0.7355,1.3683e-74}{0.7405,8.009e-71}
{0.7455,4.4781e-67}{0.7505,2.4284e-63}{0.7555,1.2932e-59}{0.7605,6.822e-56}{0.7655,3.5838e-52}{0.7705,1.8795e-48}
{0.7755,9.852e-45}{0.7805,5.1633e-41}{0.7855,2.7059e-37}{0.7905,1.418e-33}{0.7955,7.4311e-30}{0.8005,3.8943e-26}
{0.8055,2.0408e-22}{0.8105,1.0695e-18}{0.8155,5.6046e-15}{0.8205,2.9371e-11}{0.8255,1.5392e-07}{0.8305,0.00079913}
{0.8355,0.20141}{0.8405,0.2498}{0.8455,0.25}{0.8505,0.25}{0.8555,0.25}{0.8605,0.25}
{0.8655,0.25}{0.8705,0.25}{0.8755,0.25}{0.8805,0.25}{0.8855,0.25}{0.8905,0.25}
{0.8955,0.25}{0.9005,0.25}{0.9055,0.25}{0.9105,0.25}{0.9155,0.25}{0.9205,0.25}
{0.9255,0.25}{0.9305,0.25}{0.9355,0.25}{0.9405,0.25}{0.9455,0.25}{0.9505,0.25}
{0.9555,0.25}{0.9605,0.25}{0.9655,0.25}{0.9705,0.25}{0.9755,0.25}{0.9805,0.25}
{0.9855,0.25}{0.9905,0.25}{0.9955,0.25},{1.0005,0.25}
}]
\dataplot[plotstyle=line,linecolor=black]{\mydata}

\end{pspicture*}}{.9} \\
(a) $t_0=0 $&  (b) $t_1=0.31 $\\
\psxinput{\input{draw_num_1d_ex1_c}}{.9} &
\psxinput{\input{draw_num_1d_ex1_d}}{.9} \\
(c) $t_2=0.71 $&  (d) $t_3=1.29 $
\end{tabular}
\caption{Exemplary evolution of the 1D model showcasing a
turnaround behavior due to localised perception of information
[density $\rho$ solid ($-$), speed $v=-f(\rho)\mathcal{P}[\nabla
\varphi]$ dashed ($- -$) and directional conviction
$\phi_R-\phi_L$ dotted ($\cdot \cdot$)]. (a) Piece-wise constant
initial density. Part of the left crowd initially decides to move
right in order to avoid the high-density jam. (b) The separated
block moves to the right in a rarefaction-wave manner. (c) The
wave is again separated as the high-density jam gets out of sight
for centrally located pedestrians, who hence prefer the left exit
and turn. (d) The turnaround is complete and remaining pedestrians
will exit on the left.} \label{fig-1Dturnaround}
\end{figure}
This new behavioural pattern is entirely consistent with the idea
of constant re-evaluation of the optimal path based on restricted
information and cannot be observed in the original Hughes' model.
We note that without the smoothening properties of the model
around points of equal costs one obtains strong oscillations in
the turning behavior, which causes severe numerical problems. The
exact parameters of the simulation can be found in Appendix \ref{a:1d}
\subsection{2D corridor - microscopic model}
\label{s:num2dmicro}
\begin{figure}
\begin{center}
\subfigure[Time $t=0$]{\includegraphics[width=0.325 \textwidth]{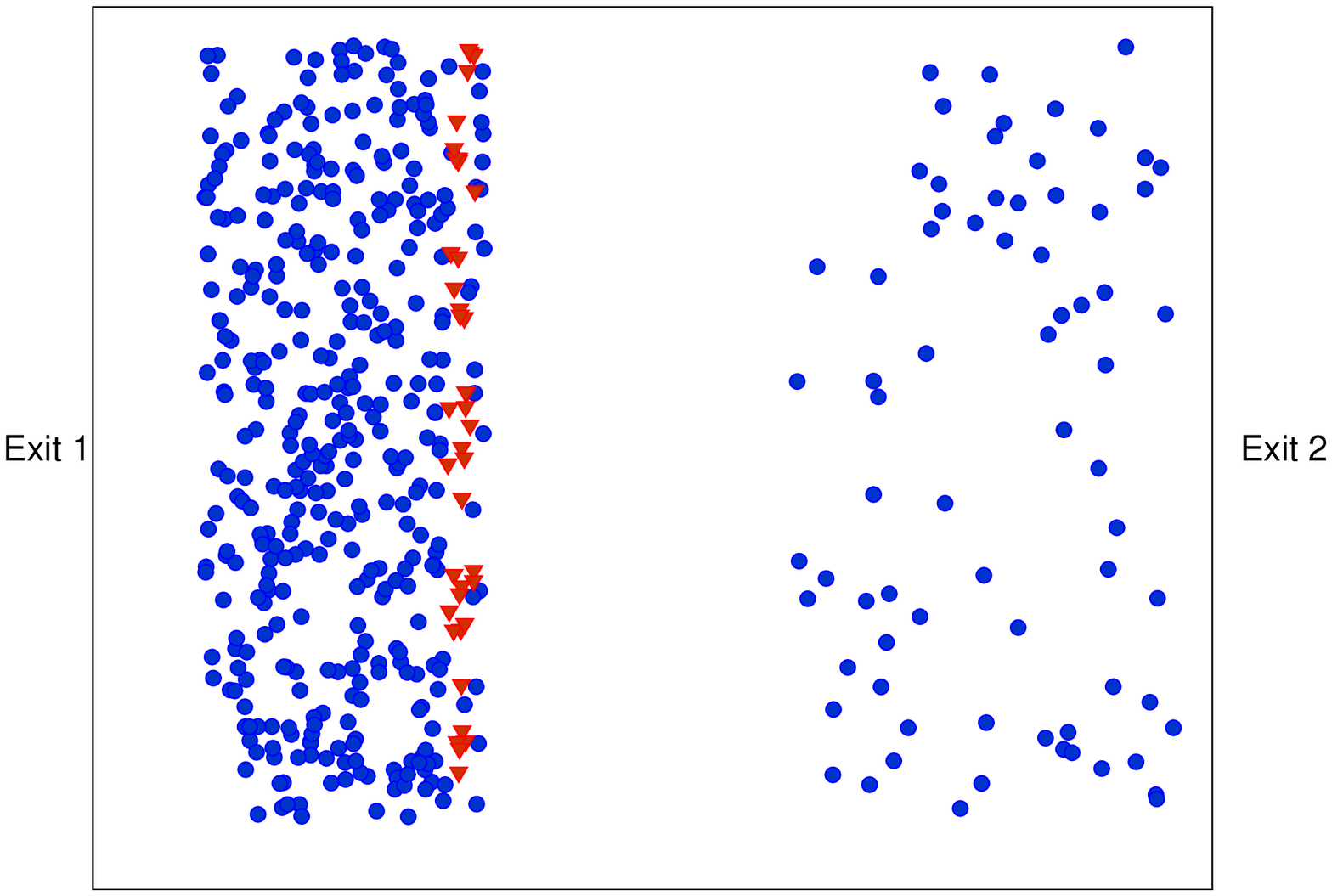}}
\subfigure[Time $t=0.2$]{\includegraphics[width=0.325 \textwidth]{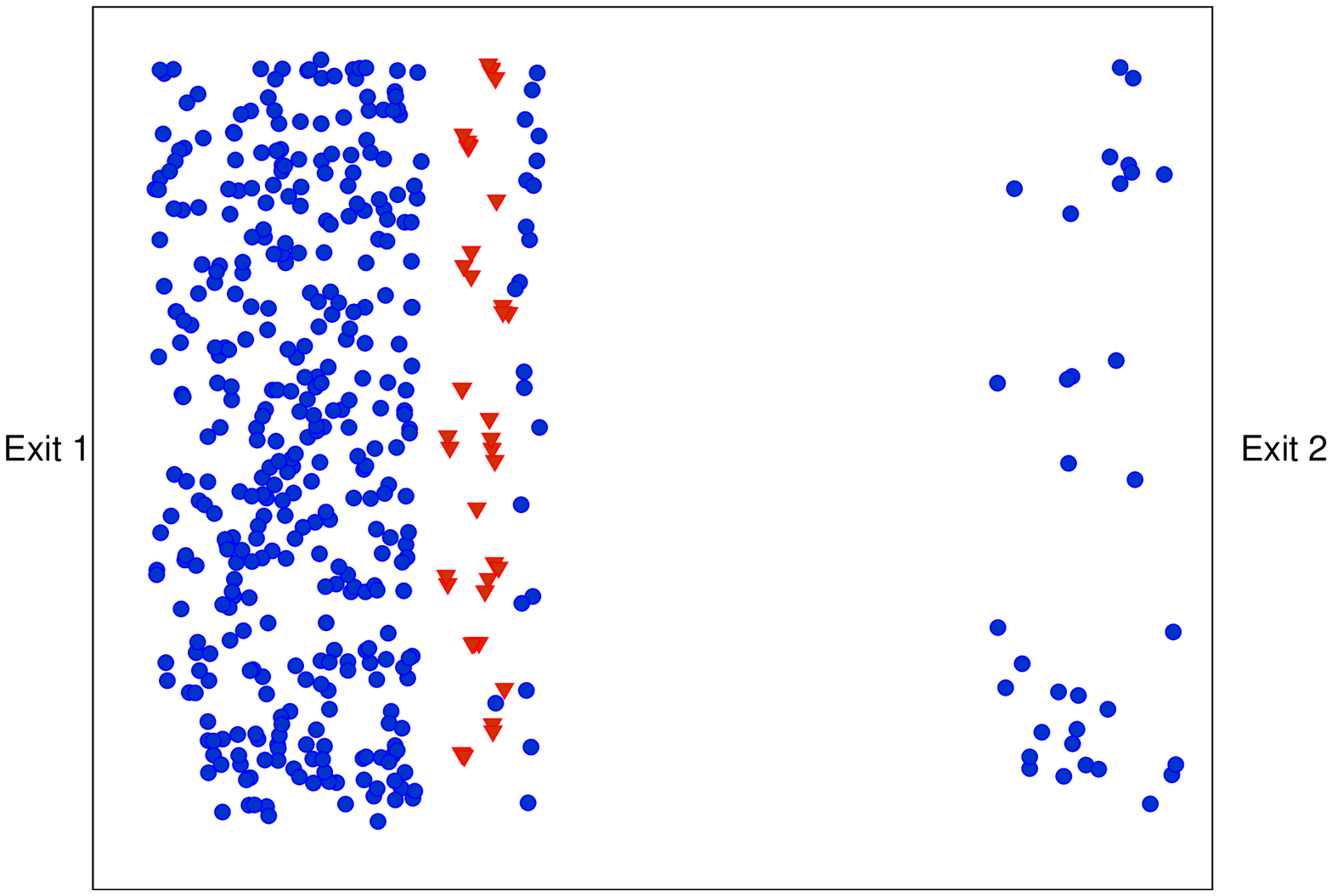}}
\subfigure[Time $t=0.4$]{\includegraphics[width=0.325 \textwidth]{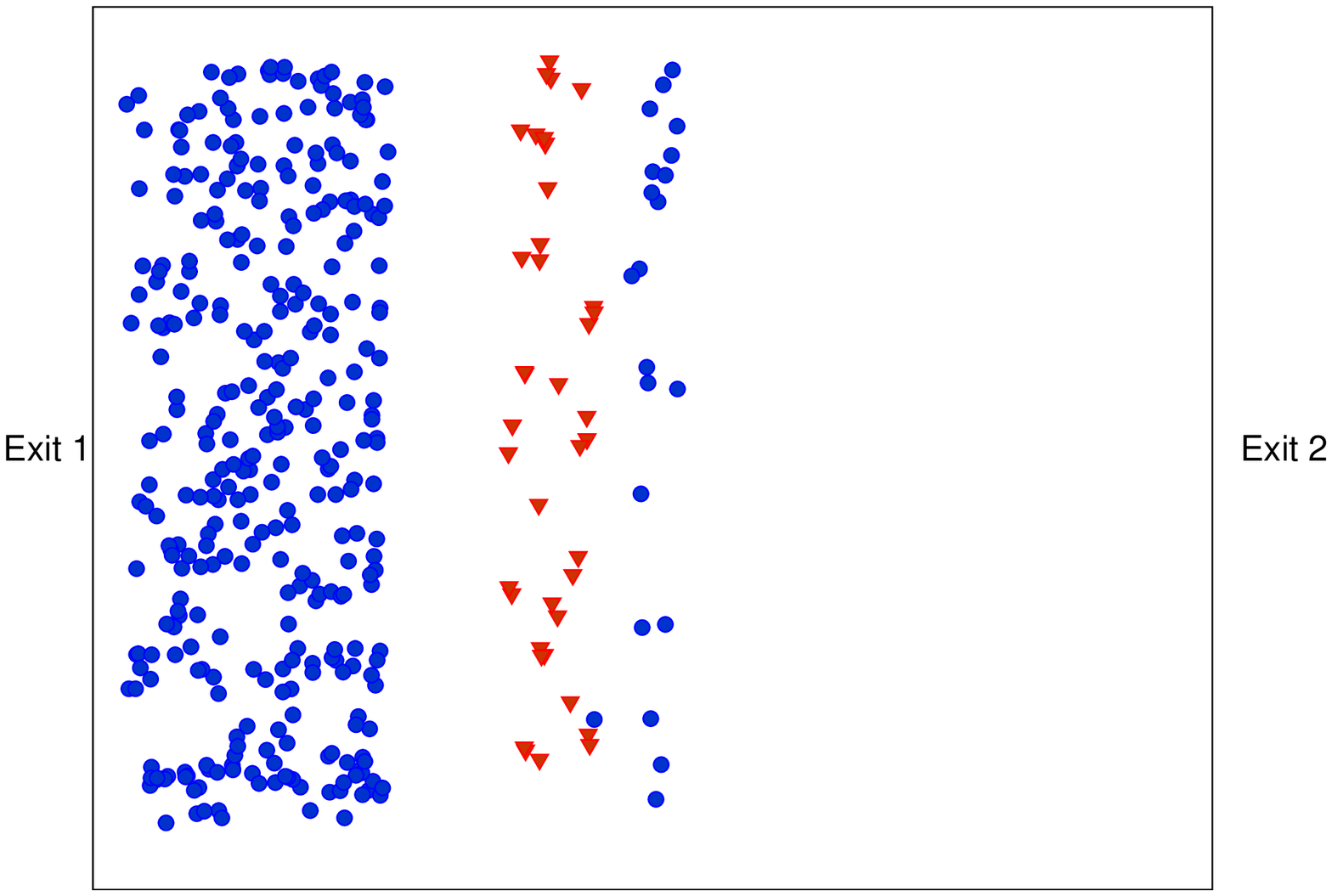}}\\
\subfigure[Time $t=0.6$]{\includegraphics[width=0.325 \textwidth]{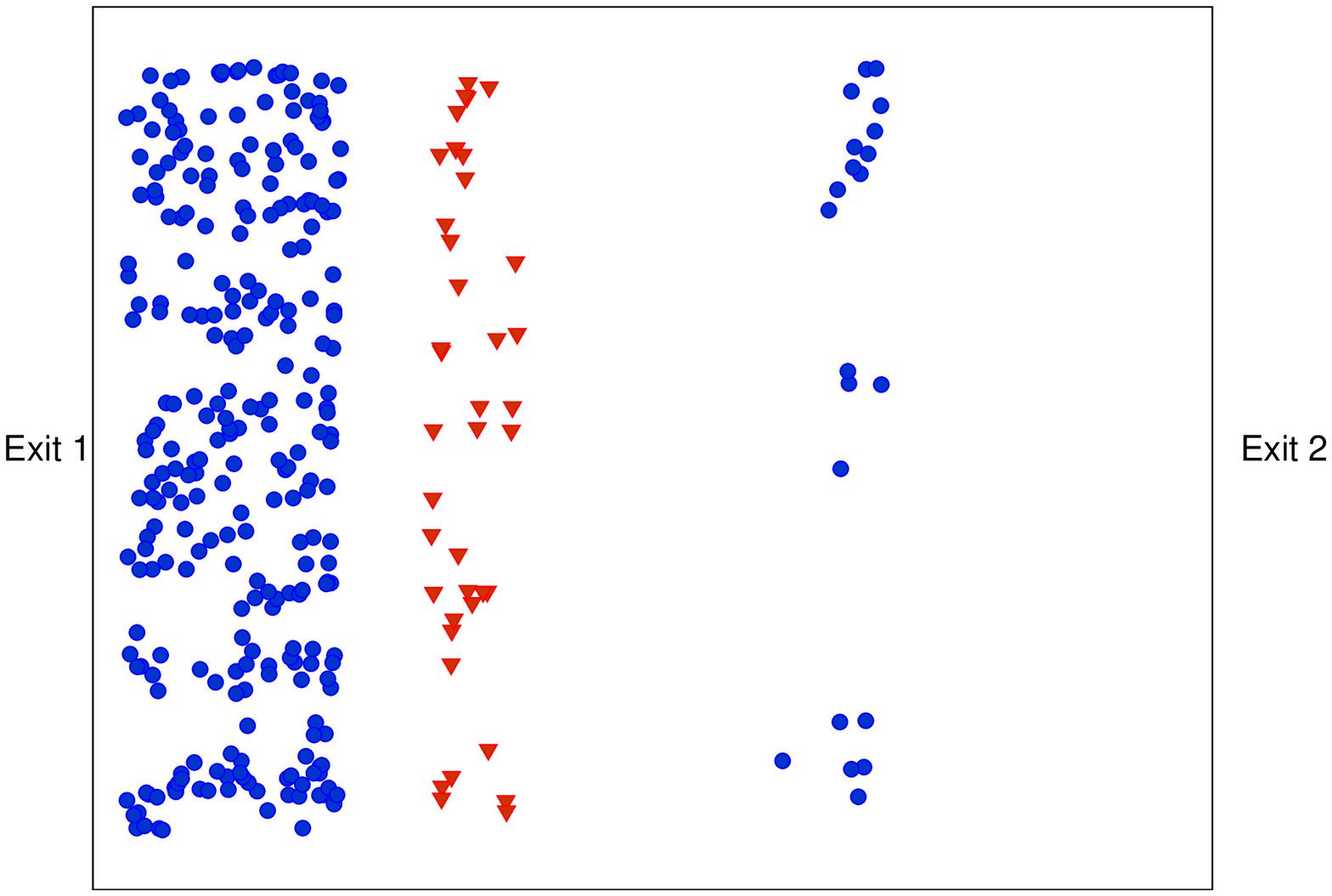}}
\subfigure[Time $t=0.8$]{\includegraphics[width=0.325 \textwidth]{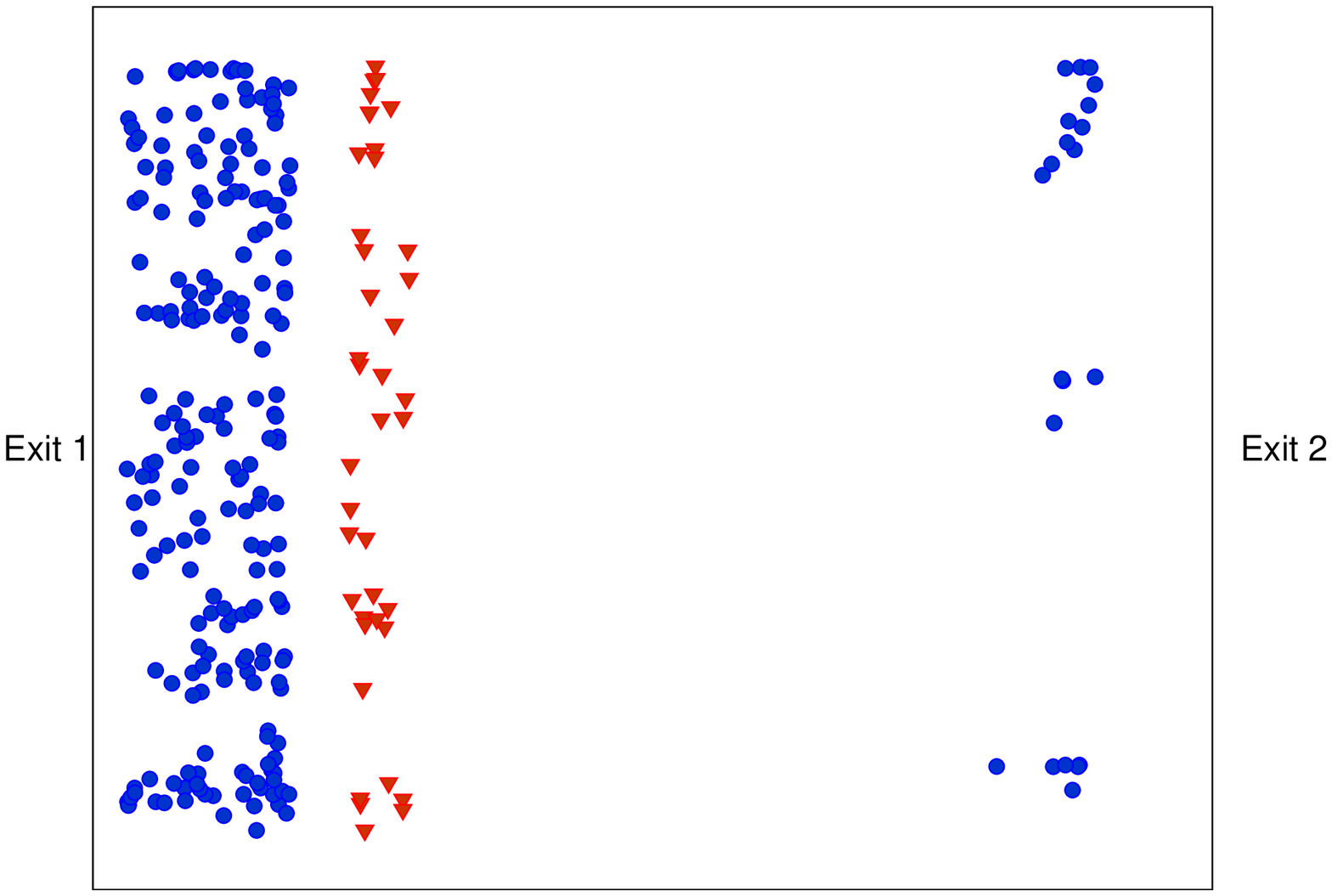}}
\subfigure[Time $t=1$]{\includegraphics[width=0.325 \textwidth]{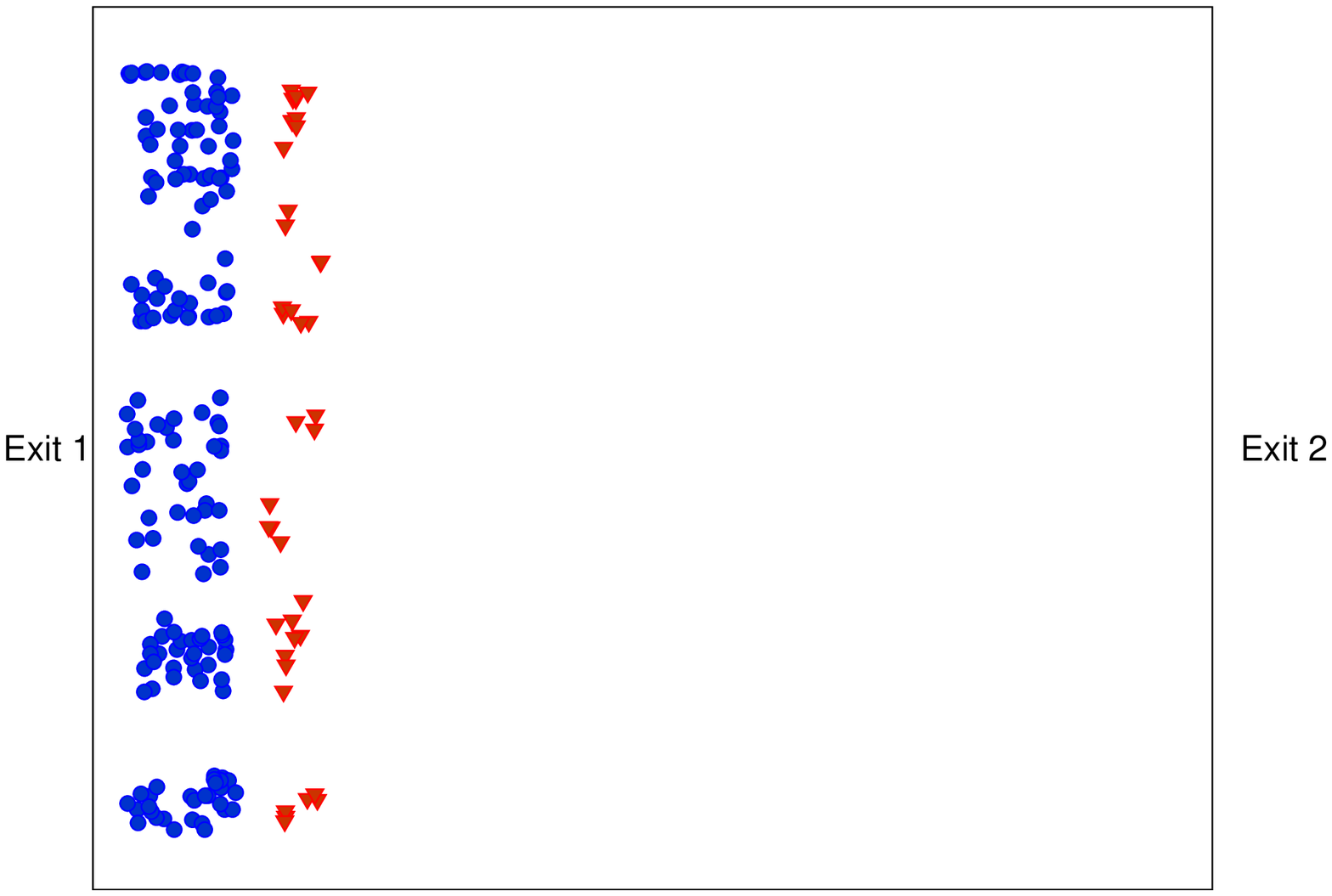}}\\
\subfigure[Exit percentage]{\includegraphics[width=.45 \textwidth]{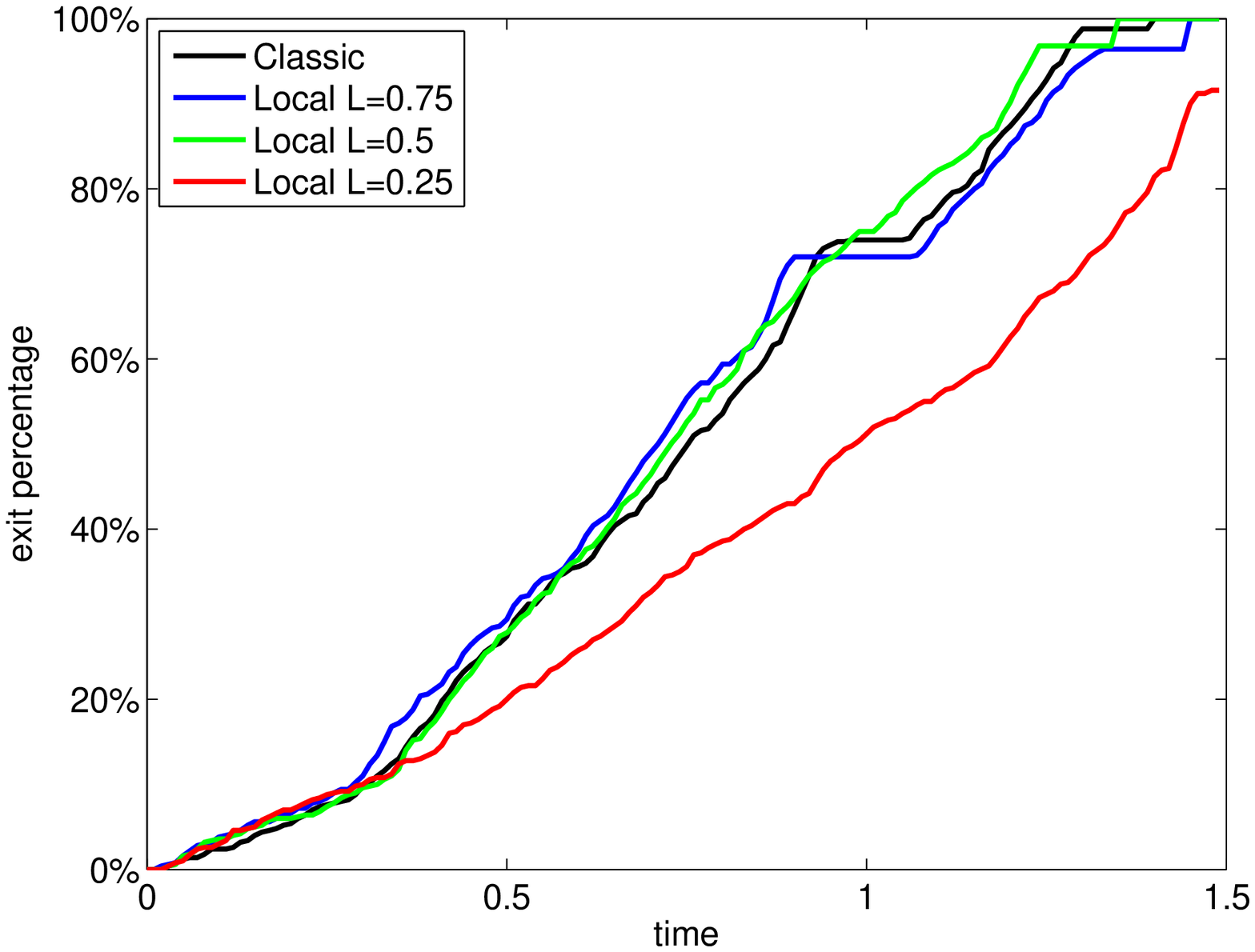}}
\end{center}
\caption{Dynamics of the microscopic model in a two-dimensional
corridor with two exits at the right and left. We observe a
similar behavior as in the 1D simulation in \ref{s:1dcorrmac} -
individuals (visualized by red triangles) initially decide to move
to the more distant exit, but after the congestion at the resolves
in time, they turn around and take the closer
exit. (a) - (f): particle solution for different times and $L=0.25$, (g): exit percentage over time for different values of L. }
\label{f:corridor_micro}
\end{figure}
We illustrate the dynamics of the microscopic model in a
two-dimensional symmetric corridor $\Omega = [0,1] \times [0,
\frac{1}{2}]$  with exits at the left and right side, i.e. $x = 0$
and $x = 1$. The 1D case of section \ref{s:1dcorrmac} can be
interpreted as a projection of this two-dimensional geometry. We
consider the same initial distribution of individuals, i.e. the
positions of all $500$ particles are distributed according to the
initial pedestrian density  \eqref{e:initsymcorridor}. For
$L=0.25$, Figure \ref{f:corridor_micro}(a)-(f) nicely illustrates
a similar turn-around behavior as in the 1D macroscopic
simulations. At the beginning the group close to the left exit
splits, one part exits through the left exit the other one moves
towards the more distant right exit. As the density close to the
left exit decreases in time, the group moving towards the more
distant exit splits again, i.e. parts of the group turn around and
move back again. We marked all individuals, which initially moved
towards the right but then turn around, with red triangles.
Furthermore, Figure \ref{f:corridor_micro}(g) shows the change of
the evacuation performance for different sizes of the local vision
cones $L$. Here we plot the percentage of the total initial mass
outside the domain versus time. Decreasing $L$ and hence the
perceived information, we observe that the overall evacuation
performance first is merely diminished, and only begins to drop
significantly after a certain threshold. The evacuation time will
approach the uninformed eikonal case $L=0$, which is not shown.
All parameters can be found in Appendix \ref{a:2dmicro}.

\subsection{2D non-symmetric corridor - macroscopic model}
\label{s:num2dmacro}
Now we turn to the macroscopic model in two dimensions. Again we
consider the corridor $\Omega=[0,1]\times[0,\frac{1}{2}]$, the
exits however form only a part of the left and right edges, hence
we obtain a fully two-dimensional dynamics where boundary
conditions matter. The left exit is located between $(0,0)$ and $(0,0.1)$ and the right exit is the segment connecting $(1,\tfrac{1}{2})$ and $(1,0.4)$.
The initial density Figure \ref{fig-2dmacroglobal}(a) is given as a low density group of pedestrians on the left and a high density group on the right
\begin{equation}
\rho_0(x,y) = \begin{cases}
0.1 & 0.05\leq x \leq 0.3 \,,\, 0\leq y \leq 0.25, \\
0.95 & 0.6\leq x \leq 0.95,\\
0 & \text{otherwise}.
\end{cases}
\end{equation}

\begin{figure}
\begin{center}
 \subfigure[Time $t=0$]{ \includegraphics[height=.205\textwidth, width=.4325\textwidth]{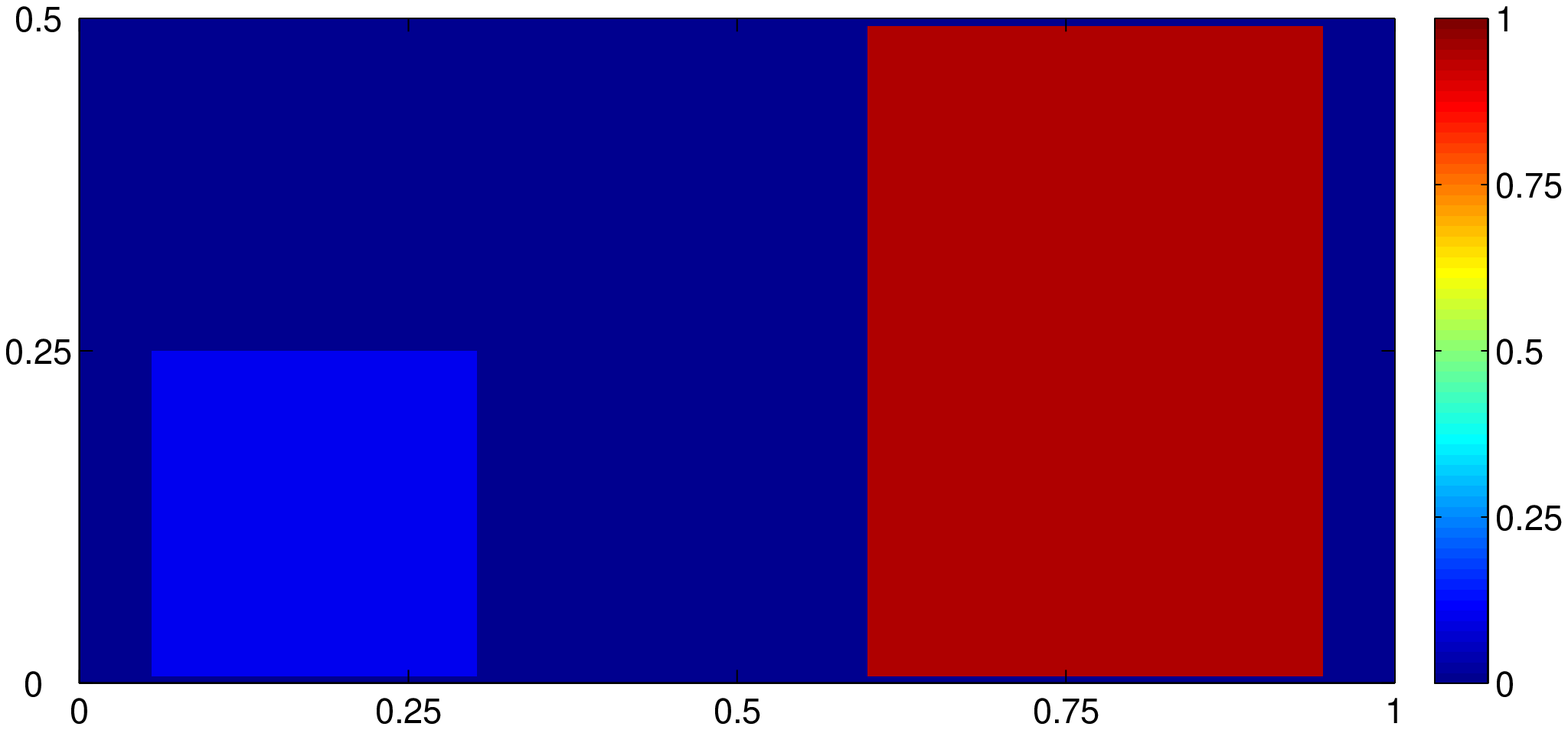} }\hspace{1mm}
\subfigure[Time $t=0.8$]{\includegraphics[keepaspectratio=true,width=.45\textwidth]{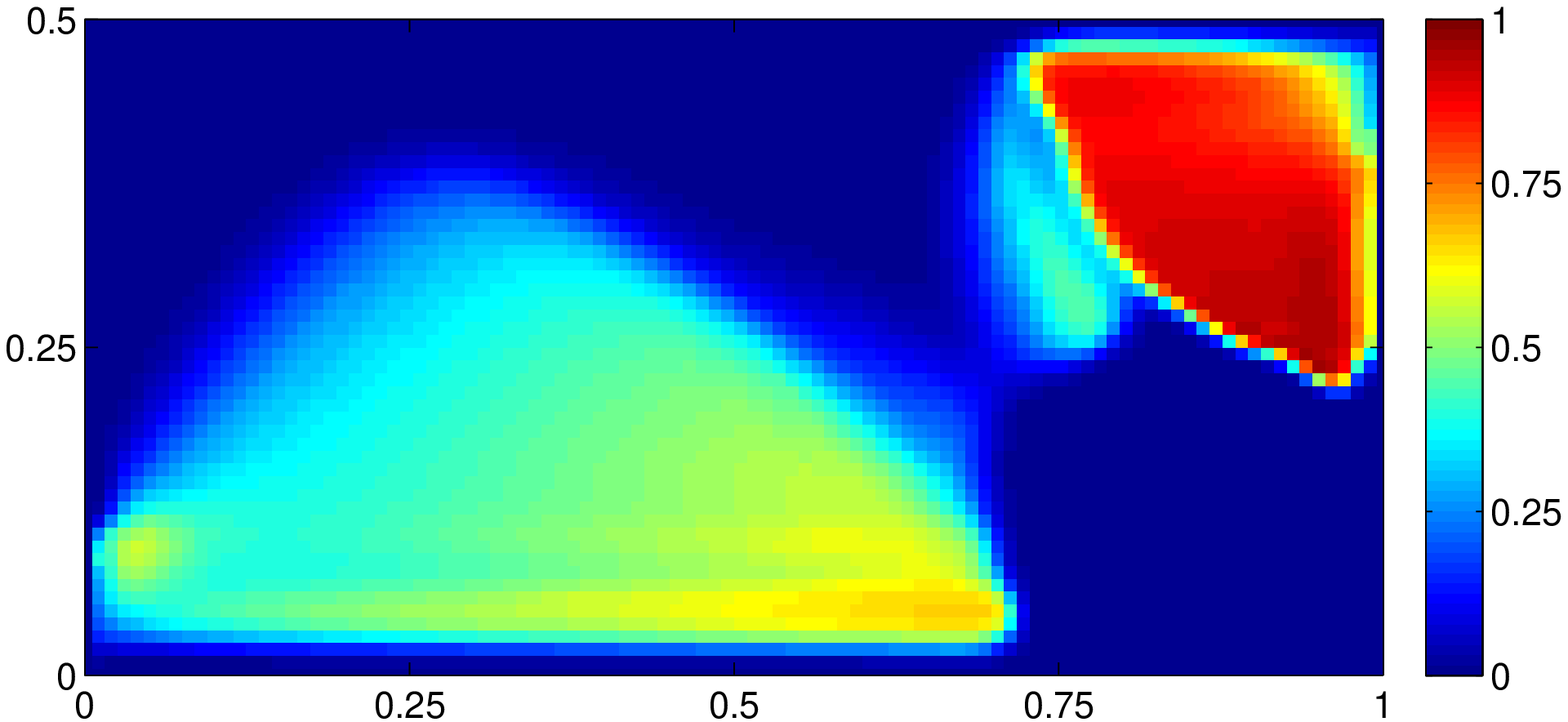}}\\
\subfigure[Time $t=1.07$]{\includegraphics[keepaspectratio=true,width=.45\textwidth]{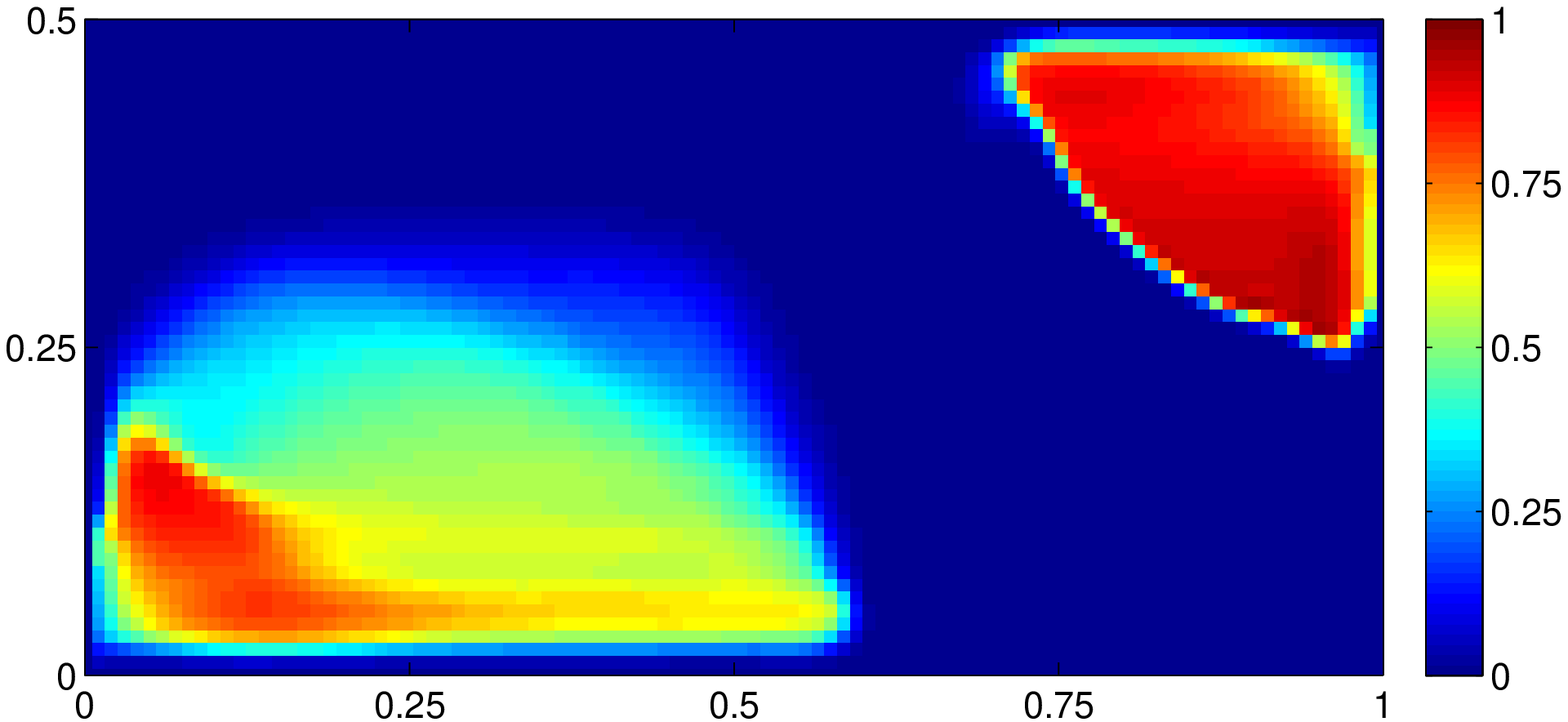}}
\subfigure[Time $t=1.4$]{\includegraphics[keepaspectratio=true,width=.45\textwidth]{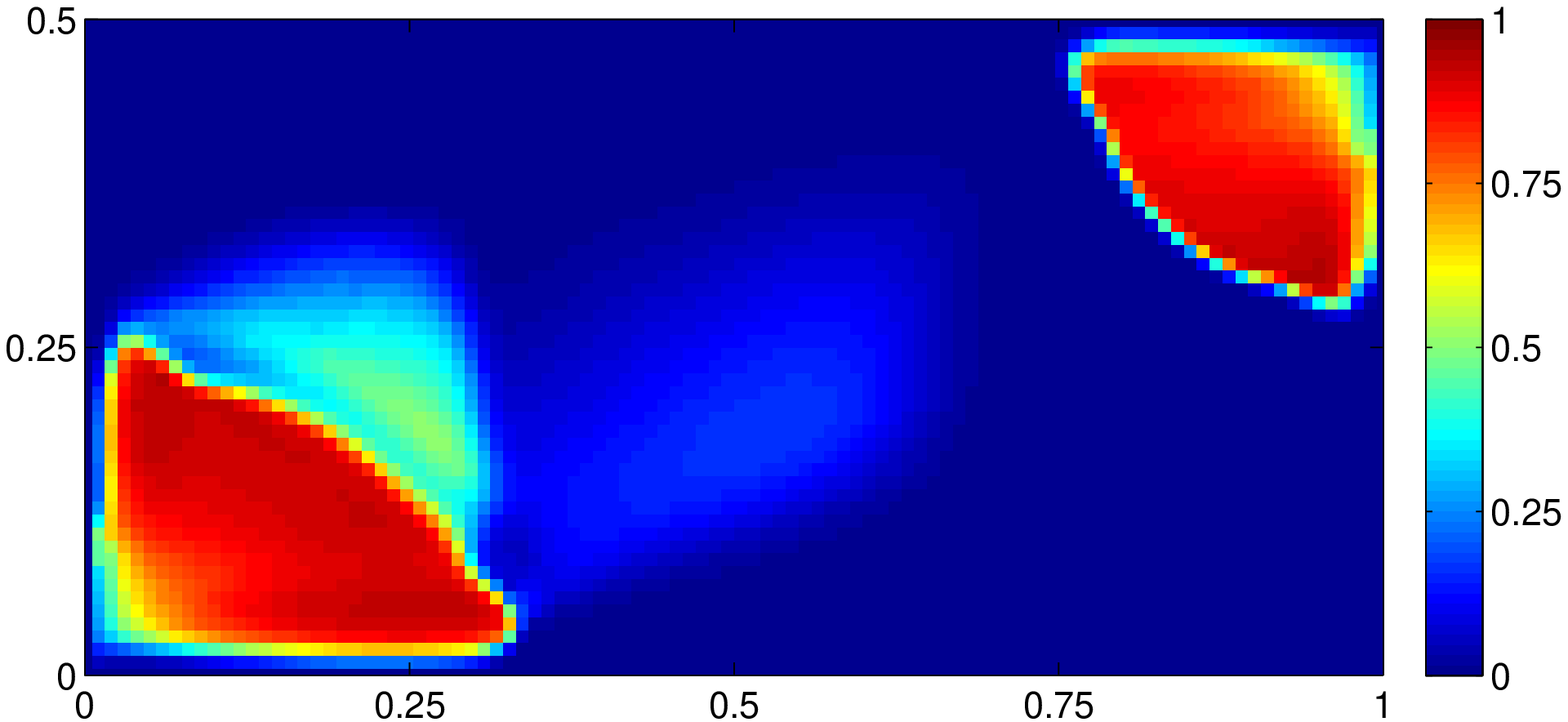}}\\
\subfigure[Time $t=2.1$]{\includegraphics[keepaspectratio=true,width=.45\textwidth]{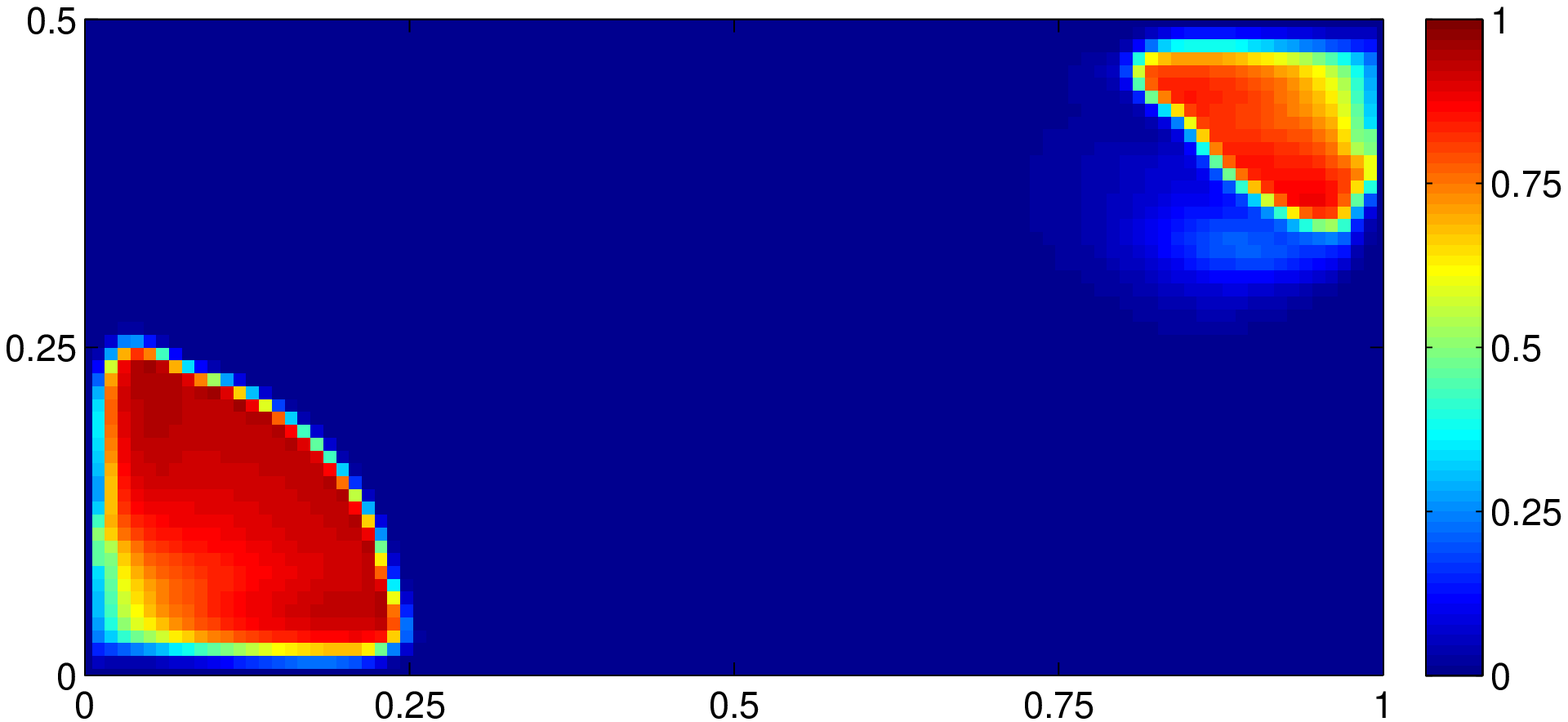}}
\subfigure[Time $t=2.75$]{\includegraphics[keepaspectratio=true,width=.45\textwidth]{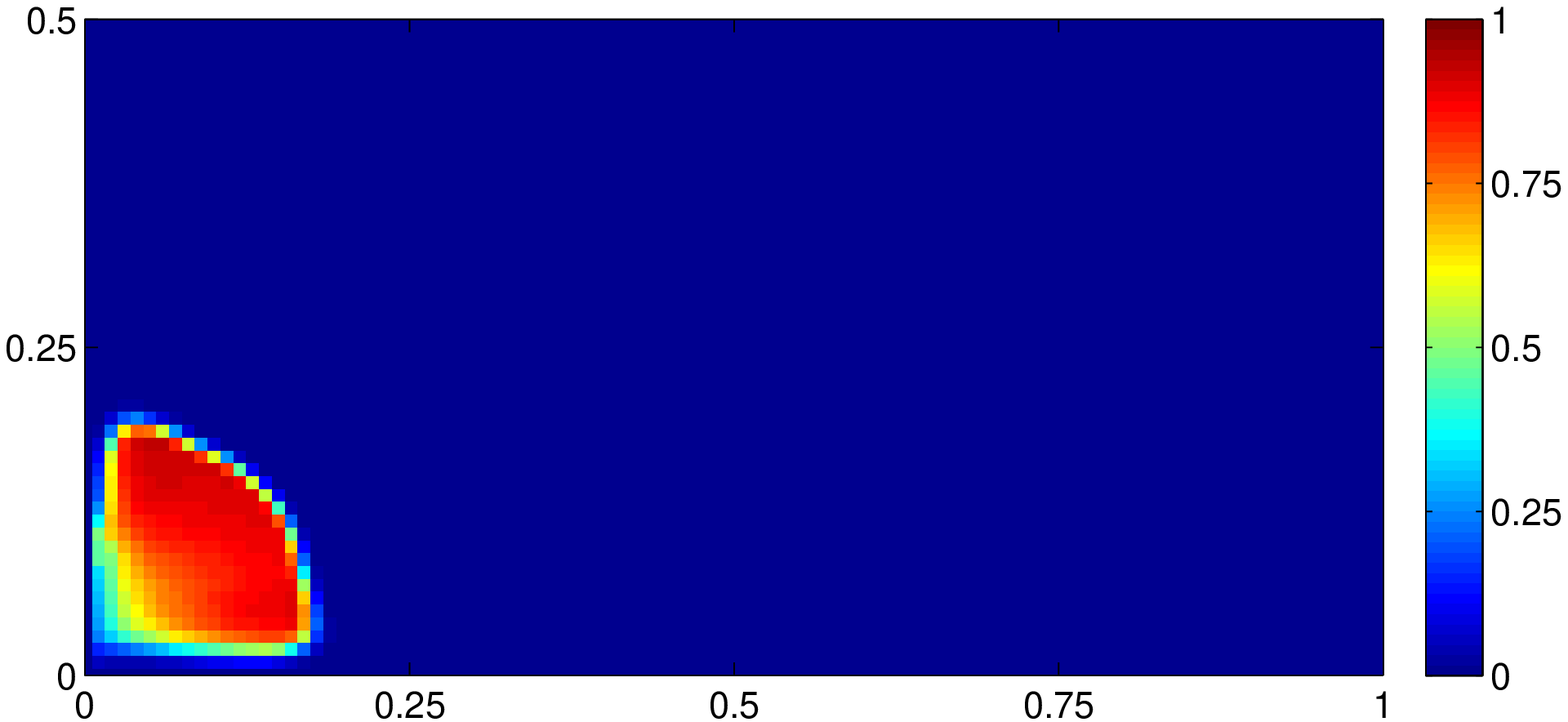}}\\
\end{center}
\caption{Two-dimensional macroscopic dynamics : We simulate model \eqref{e:localmodel} with global visual perception $V_x=\Omega$. Two groups of pedestrians are initially placed in a corridor with a lower left and an upper right exit. The high density group separates according to the path optimization mechanism, as illustrated in several time snapshots of the density in (a) - (f). The right exit is vacated before the left exit and  the final evacuation time (not shown here) is $\approx 3.4$.}
\label{fig-2dmacroglobal}
\end{figure}

We first study the case of global vision $L=\infty \Leftrightarrow
V_x=\Omega$ in Figure \ref{fig-2dmacroglobal}. In (b), the low
density group turns towards left and is quickly vacated. The high
density group on the other side splits along a curve of sonic
points. Pedestrians turning to the right  cause a jam in front of
the right exit, whereas left-turning pedestrians occupy the
corridor in a rarefaction-type manner inherited by the physical
flux law \eqref{fluxlaw}. Upon arrival at the left exit,
pedestrians pile up and form a new jam (c). Hence, a fraction of
the density turns around again and heads for the right exit (d,
around $(0.5,.25$)), having to cross most of the corridor again
(e). However, most of the pedestrians are committed to the left
exit and do not turn, because the severeness of the left jam does
not compensate their expected travel time, and the left exit is
vacated later than the right exit (f).

Compared to the classical Hughes model, the relaxation term
\eqref{e-relaxation} and the conviction-based interaction
\eqref{e:localinteraction}-\eqref{e:conviction} allow for a smooth
turning behaviour. The wall-repulsion \eqref{e:wall} causes a
density gap, which is different to zero-flux conditions as in
\cite{HuangonHughes2009}, but prevents any spurious effects of the
boundary flux.

\begin{figure}
\begin{center}
 \subfigure[Time $t=0.25$]{ \includegraphics[keepaspectratio=true,width=.45\textwidth]{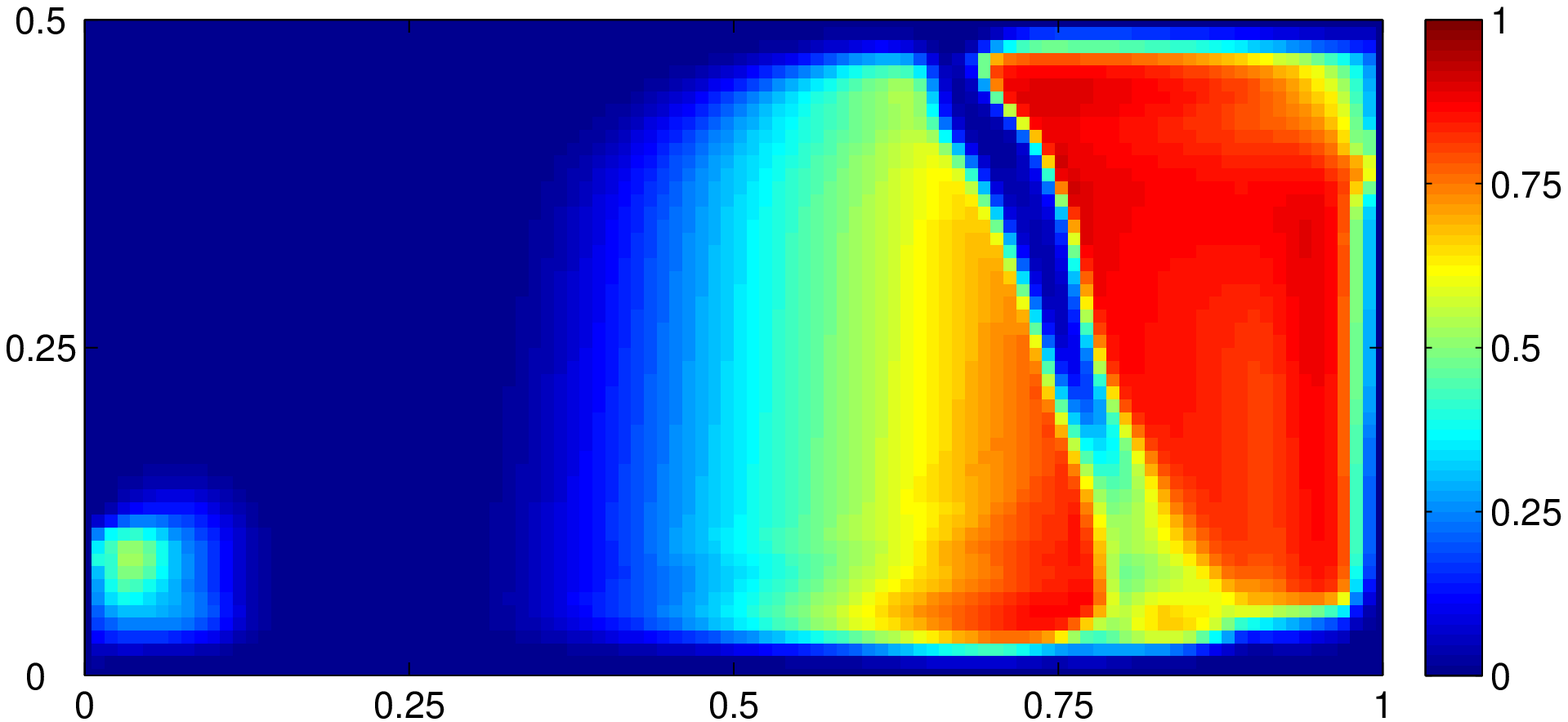}}
 \subfigure[Time $t=0.8$]{ \includegraphics[keepaspectratio=true,width=.45\textwidth]{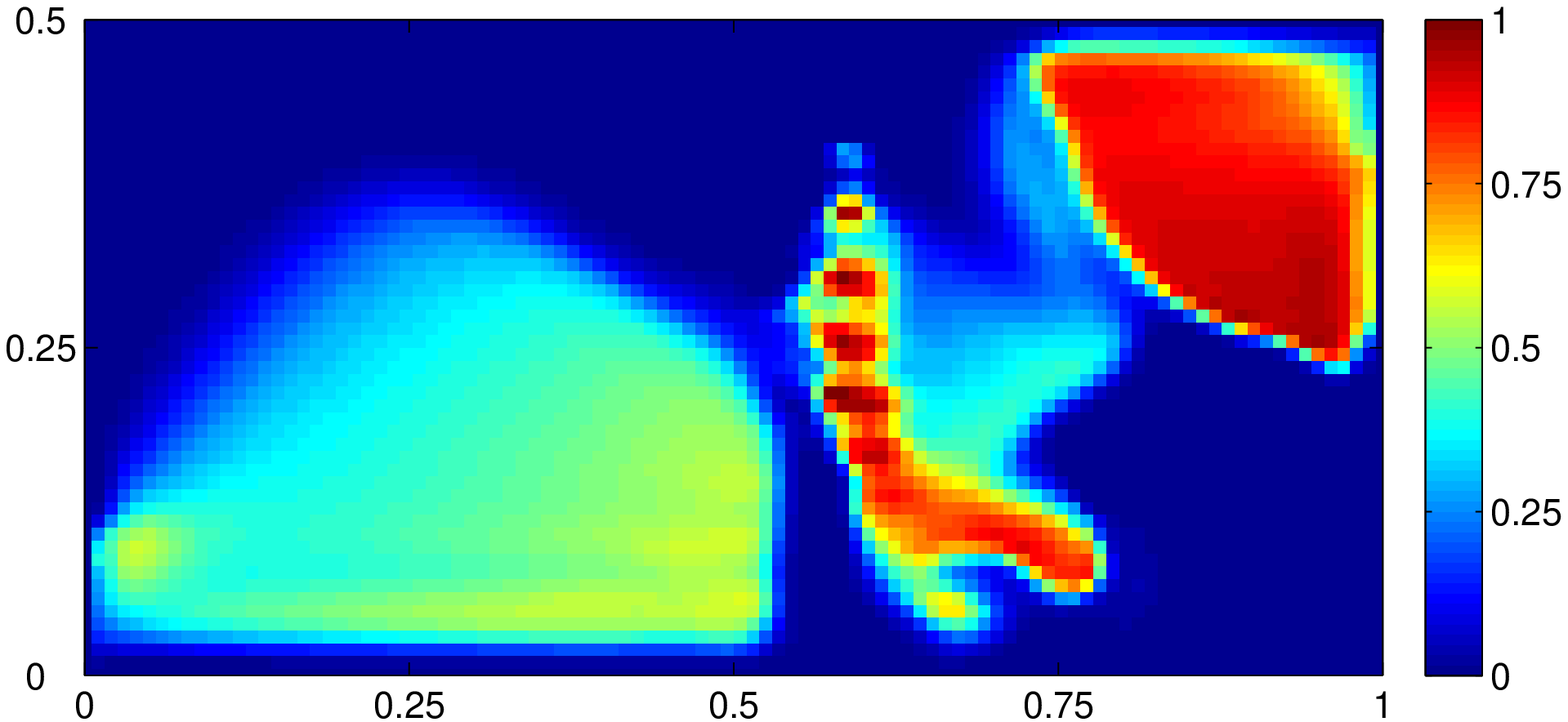}}\\
 \subfigure[Time $t=1.07$]{ \includegraphics[keepaspectratio=true,width=.45\textwidth]{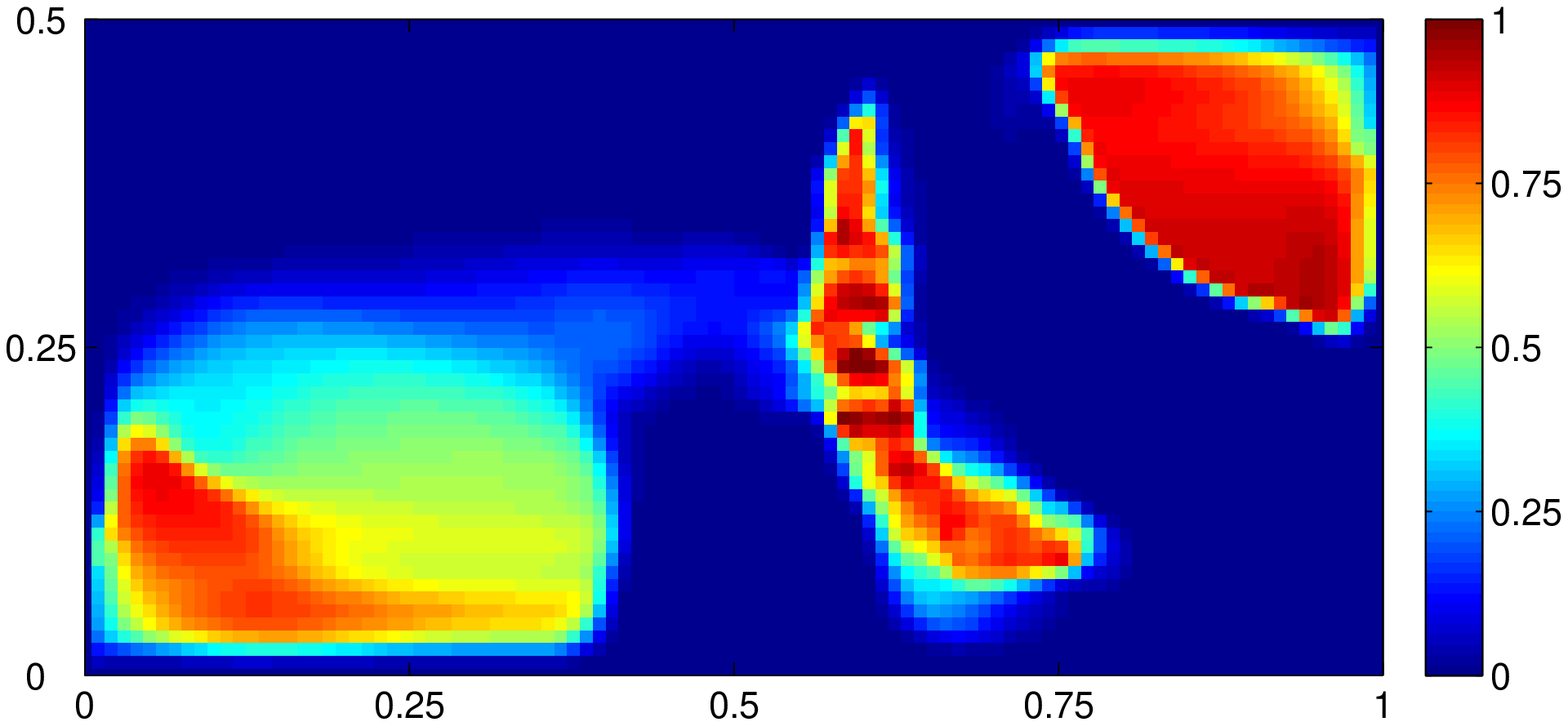}}
 \subfigure[Time $t=1.4$]{ \includegraphics[keepaspectratio=true,width=.45\textwidth]{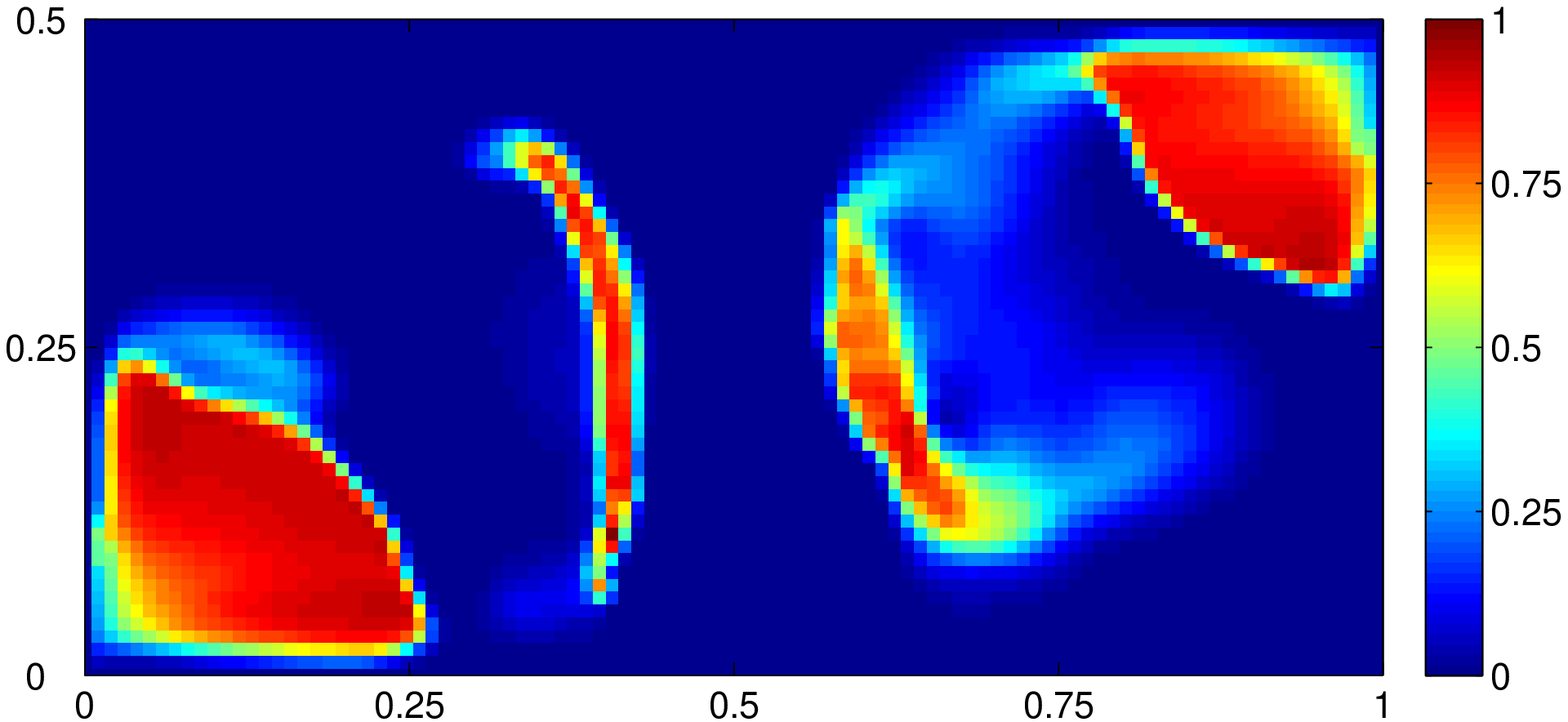}}\\
 \subfigure[Time $t=2.1$]{ \includegraphics[keepaspectratio=true,width=.45\textwidth]{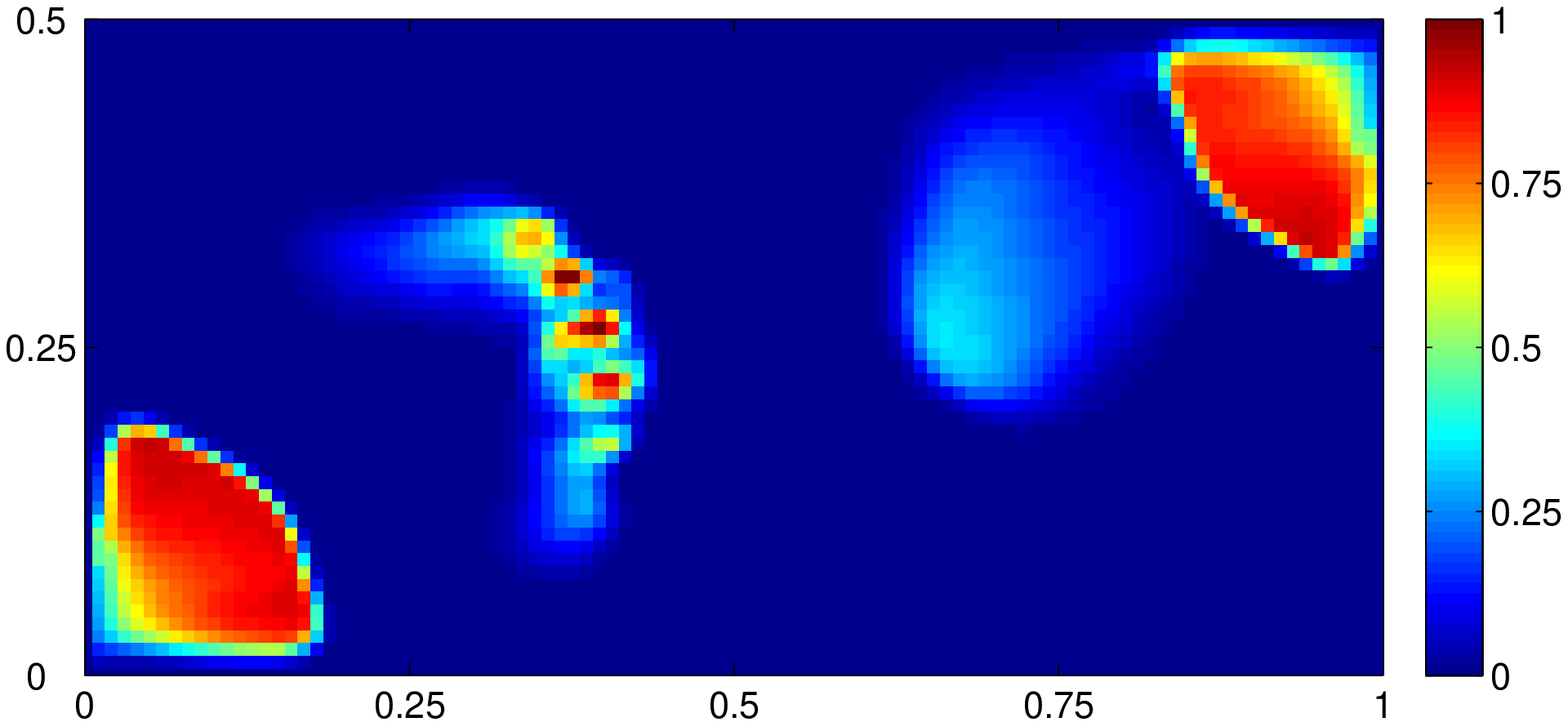}}
 \subfigure[Time $t=2.75$]{ \includegraphics[keepaspectratio=true,width=.45\textwidth]{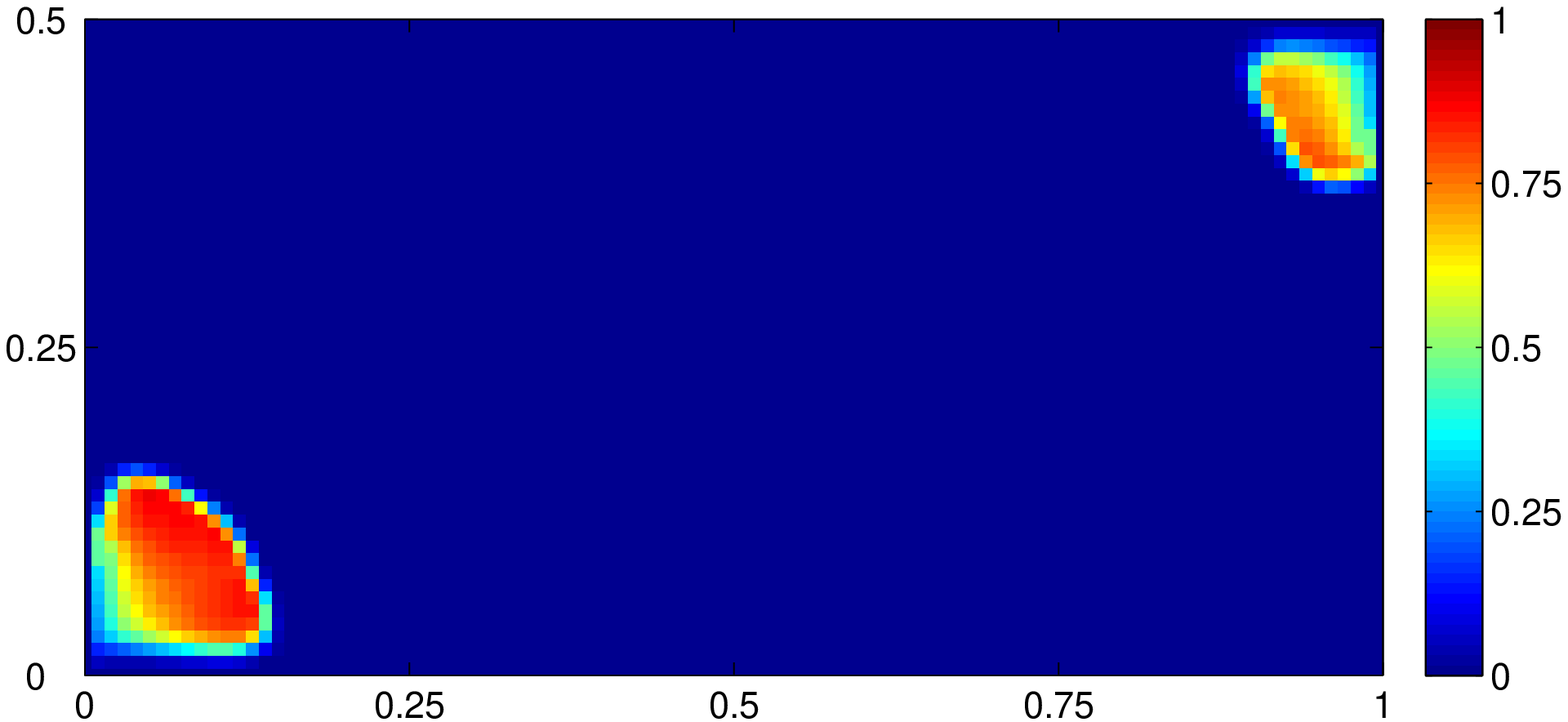}}\\
\end{center}
\caption{Two-dimensional macroscopic dynamics : We simulate model \eqref{e:localmodel} with local visual perception $L=0.75$ and the same initial configuration as Figure \ref{fig-2dmacroglobal}. In the time snapshots of the density, we observe again an initial separation of the high density group (a). A waiting phenomena of up to two local groups is clearly observed (b-d). Crucially, waiting pedestrians are able to choose their exit later (d-e), which leads to a rather simultaneously clearing of both jams (f), as opposed to Figure \ref{fig-2dmacroglobal}. The final evacuation time (not shown here) is $\approx 3.025$.}
\label{fig-2dmacrolocal}
\end{figure}

It is clear that even the planning algorithm incorporated into the
classic Hughes' model does not lead to an optimal evacuation. One
reason for suboptimality in Figure \ref{fig-2dmacroglobal} is that
pedestrians \emph{have to} keep in motion constantly but cannot
predict the occurrence of future jams. Hence, pedestrians are
likely to walk towards an exit that will be blocked in the future,
as seen in the example.

In Figure \ref{fig-2dmacrolocal} we study the same initial
configurations with localised perception and a radial vision cone
of diameter $L=0.75$. The initial separation phase (a) is similar
to Figure \ref{fig-2dmacroglobal}. As pedestrians move from the
right to the left, the right jam gets out of sight and its
influence diminishes. At the same time, the density on the left
becomes visible. At a certain point a balance is achieved and
pedestrians locally accumulate around an area of equal walking
costs, where in this case they are able to stop (b-c). Hence, we
observe a waiting behavior which cannot be observed in classical
Hughes' type models. Looking from (c) to (d), a high density jam
forms at the left exit, which causes part of the left-walking
pedestrians to turn right after enough conviction is gathered.
Together with some outflow of the first waiting group, a second
waiting group is formed (d). Pedestrians in a waiting group choose
to move if one direction becomes favorable. As both jams at the
exits reduce at the same rate, the left waiting group walks to the
left and vice versa (e). Finally, the waiting groups dissolve and
the exits get vacated at a rather similar time (f), and the
evacuation time improves compared to Figure
\ref{fig-2dmacroglobal}.

\begin{figure}
\begin{center}
\includegraphics[width=0.6\textwidth]{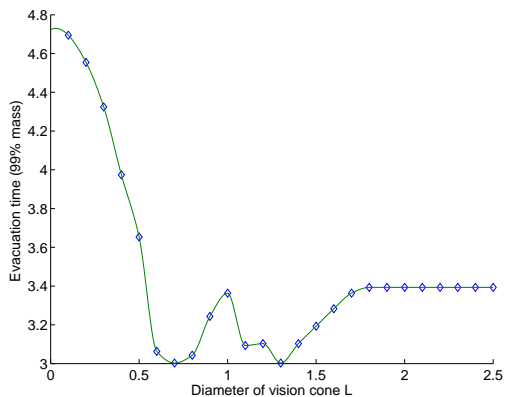}
\end{center}
\caption{Evacuation time of the macroscopic model for varying vision cone diameters $L$: The performance can improve with limited vision. $L=0$ corresponds to the eikonal case whereas $L>2.5$ implies unlimited vision in the given corridor.}
\label{fig-2dmacroevac}
\end{figure}

The fact that evacuation performance can improve under limited
perception of information is surprising at first glance. Our
simulations give an good explanation for the phenomena: As
pedestrians show a waiting behaviour, they are less likely to be
trapped in the jam arising at the left exit. In fact, the waiting
is made possible by the combined effect of multiple sonic points
due to local vision and the smoothed turning mechanism. Naturally,
this is not generally the case and cannot be a-priori answered.
For small vision lengths $L$, the dynamics will converge to the
velocity field given by the eikonal equation, which is our initial
configuration will exit almost all pedestrians using the right
exit and perform poorly. In Figure \ref{fig-2dmacroevac}, we study
the evacuation time of $99\%$ of the initial mass as a function of
the diameter $L$. We unexpectedly find in this case wo optimal
values of $L$ for which the evacuation time is minimal. The
classical Hughes' evacuation time ($L$ large) is always less or
equal than the eikonal case ($L=0$), however there is no way to
generally argue that there will always be a minimum in between.

\section{Conclusion}\label{s:conclusion}\label{s:conclusions}
In this work we introduced a localized smooth variant of Hughes's
model for pedestrian crowd dynamics. We regularised the original
model, composed by an eikonal equation and a continuity equation.
First by a local interaction term, which intermediates individual
pointwise path optimisation towards conviction-weighted walking
directions. Secondly, we allowed pedestrians to stop, if they are
undecided, using a smooth approximation of the normalization
condition. Most importantly, we restrict the information on the
global density each pedestrian can use for her planning algorithm
to a local surrounding area. This is a very realistic assumption
for large crowds that has not been considered in the literature so
far. We presented both a microscopic and a macroscopic version,
and illustrated the model components in the one-dimensional case.
In terms of analytical results, a rigorous theory for these kind
of equations in multiple dimensions is currently out of reach to
the best of our knowledge. However, we were able to identify some
qualitative properties of the dependence of the optimal path on
the vision cone that allow for a reduction of complexity.

The numerical approximation of the model on both levels has been discussed and utilizes several techniques including sweeping and marching methods, particle approximations and finite volume schemes. Though the numerical costs of computing a solution have increased due to the inhomogeneity of vision cones, we observe new effects and phenomena in the model based on our simulations.
First, local groups of pedestrians are able to change repeatedly their walking direction towards an exit. This 'multiple turn-around behaviour' can explained by the multiple sonic points of the estimated walking costs, which by construction cannot occur in the classical case.
We stress that the smoothening and conviction terms are crucial to allow a swift turning behavior, which is not trivial to model in first order equations.
Second, the model replicates a waiting behaviour in case of undecided pedestrians, i.e. in areas where locally estimated walking costs towards different exits are equal.
Surprisingly, we found that this waiting phenomena induced by localized information can improve the overall evacuation performance of the crowd. In our numerical example
we observed two local minima when varying the vision cone diameter.

To conclude, we have demonstrated that local vision effects can be implemented into first order models for crowd dynamics. This leads to new unforeseen phenomena and complex behavior, whose partial understanding via qualitative properties is important for the applicability of such equations to social-economic problems. On the other hand, this work illustrates the limitations to first order models such as Hughes', where planning decisions are instantaneously updated and no social or cognitive memory is taken into account. From our point of view Hughes' type equations constitute an important building block for crowd models and a mathematically important object of study, but it cannot be expected to be fully realistic.
\section*{Acknowledgment}

JAC acknowledges support from projects MTM2011-27739-C04-02 and
the Royal Society through a Wolfson Research Merit Award. JAC and
SM were supported by Engineering and Physical Sciences Research
Council (UK) grant number EP/K008404/1. MTW acknowledges financial
support from the Austrian Academy of Sciences \"OAW via the New
Frontiers Group NSP-001.
Preprint of an article submitted for consideration in Mathematical Models and Methods in Applied Sciences, \copyright 2015  World Scientific Publishing Company, http://www.worldscientific.com/worldscinet/m3as.
\appendix
\section{Simulation parameters}
\subsection{1d macro parameters in Section \ref{s:1dcorrmac}}
\label{a:1d}
The macroscopic simulation in 1D was implemented in \texttt{MATLAB}. The domain $[0,1]$ was uniformly discretized with $\Delta x=10^{-4}$. The time step was set to $\Delta t = 5\cdot 10^{-5}.$
The vision cone was defined as $V_x=[x-L/2,x+L/2]\cap [0,1]$ with $L=0.75$. 
The radial interaction kernel $\mathcal{K}$ was chosen as the indicator function on the interval $[0,0.05]$. 
The smoothed projection operator was chosen as in \eqref{e-relaxation} with $\ell=0.05$ and $k=25$.
The wall repulsion $W(x)$ was neglected in 1D. Absorbing boundary conditions were applied at both exits. 
The cost function was numerically bounded at $c(\rho)\leq 10^4$. 
\subsection{2d micro parameters in Section \ref{s:num2dmicro}}
\label{a:2dmicro}
The microscopic simulations are implemented using the software package Netgen/NgSolve. The domain was discretized in $1438$ triangles, the time steps were set to
$\Delta t = 10^{-2}$, the final time to $T=1.5$. At time $t=0$ we distributed the $500$ particles according to the initial datum $\rho_0$ used in subsection \ref{s:1dcorrmac}.
The empirical density $\rho^N_g$ was calculated using Gaussians with variance $\sigma = 0.05$ for the smooth approximation $g$. The width of the local vision cone
was set to $L = 0.25$ in Figure \ref{f:corridor_micro}(a)-(e) and to $L = 0.75, 0.5$ and $L =0.25$ in Figure \ref{f:corridor_micro}(f).

\subsection{2d macro parameters in Section \ref{s:num2dmacro}}
\label{a:2dmacro}
The macroscopic simulation in 2D was implemented in \texttt{MATLAB}. The domain $\Omega = [0,1]\times[0,\frac{1}{2}]$ was uniformly discretized with $\Delta x=\Delta y =10^{-3}$. The time step was set to $\Delta t = 5\cdot 10^{-3}.$ A Fast Sweeping Method was used to solve the eikonal equations.
The vision cone was defined as $V_x=\{y: \norm{y-x}_2\leq \frac{L}{2}\}\cap \Omega$ with varying diameter $L$.  
The radial interaction kernel $\mathcal{K}$ was chosen as
\begin{equation*}
\mathcal{K}(x)=\begin{cases}\exp\left(-\frac{b^2}{b^2-\norm{x}^2}\right) & \norm{x}\leq b, \\
0 & \text{else}.
\end{cases} 
\end{equation*}
with $b=0.05$. 
The smoothed projection operator was chosen as in \eqref{e-relaxation} with $\ell=0.05$ and $k=25$.
The wall repulsion $W(x)$ was defined with the width of the boundary layer function $\chi_w$ set to $w=0.025$.
The wall costs $W(x)$ were numerically bounded at $W(x)\leq c(0.975)$. The cost function was numerically bounded at $c(\rho)\leq 10^3$. The numerical accuracy for vanishing density was set to $10^{-7}$. 


\bibliographystyle{abbrv}
\bibliography{localhughes}

\end{document}